\documentclass[12pt]{amsart}
\usepackage[margin=0.5in]{geometry}
\usepackage{graphicx, epstopdf,color}
\usepackage{amssymb,amscd,amsthm,amsxtra}
 
\numberwithin{equation}{section}

\newtheorem{theorem}{Theorem}
\newtheorem{lemma}[theorem]{Lemma}
\newtheorem{proposition}[theorem]{Proposition}
\newtheorem{remark}[theorem]{Remark}
\newtheorem{corollary}[theorem]{Corollary}
\newcommand{\R}{\mathbb R}
\newcommand{\rr}{\mathbb R}

\newcommand{\calR}{\mathcal R}
\renewcommand{\leq}{\leqslant}
\renewcommand{\le}{\leqslant}
\renewcommand{\geq}{\geqslant}
\renewcommand{\ge}{\geqslant}

\begin{document}

\author{Mariel S\'aez}
\address[Mariel S\'aez]{Facultad de Matem\'aticas, 
Pontificia Universidad Cat\'olica de Chile,
Avenida Vicu\~na Mackenna 4860, Macul, 
6904441 Santiago, Chile}
\email{mariel@mat.puc.cl}

\author{Enrico Valdinoci}
\address[Enrico Valdinoci]{Department of Mathematics and Statistics,
University of Western Australia,
35 Stirling Hwy, Crawley WA 6009, Australia, and
School of Mathematics and Statistics,
University of Melbourne,
Peter Hall Building,
Parkville VIC 3010,
Australia, and
Weierstra{\ss} Institut f\"ur Angewandte Analysis
und Stochastik, Mohrenstra{\ss}e 39, 10117 Berlin, Germany,
and Dipartimento di Matematica, Universit\`a degli studi di Milano,
Via Saldini 50, 20133 Milan, Italy, and
Istituto di Matematica Applicata e Tecnologie Informatiche,
Consiglio Nazionale delle Ricerche,
Via Ferrata 1, 27100 Pavia, Italy}
\email{enrico@mat.uniroma3.it}

\title[On the evolution by fractional mean curvature]{On the evolution by fractional mean curvature}

\begin{abstract}
In this paper we study smooth solutions to a fractional mean curvature flow 
equation. We establish a comparison principle and consequences such as uniqueness 
and finite extinction time for compact solutions. We also establish evolutions 
equations for fractional geometric objects that in turn
yield the preservation of certain 
quantities, such as the
positivity of the fractional mean curvature.
\end{abstract}
\maketitle

\section{Introduction}

In the recent literature, an intense study has been performed
on some fractional counterparts of the classical perimeter
and of the motion by mean curvature. The interest in this kind
of topics comes from several considerations. First of all,
from the theoretical point of view, the analysis of nonlocal and fractional 
operators has an ancient tradition,
which have been vividly renovated recently by new exciting
discoveries.
In particular, a notion of fractional perimeter has been introduced
in~\cite{CRS}
and its relation with a fractional mean curvature flow
was discussed in detail in~\cite{IMB, CaSo}.

Roughly speaking, given~$s\in(0,1)$ the fractional perimeter in the whole
of~$\R^n$ of a bounded set~$E$ may be seen as
the seminorm in~$H^{s/2}(\R^n)$ of the characteristic function of~$E$
(and this notion may be also localized inside a bounded domain~$\Omega\subset\R^n$). The first variation of
the fractional perimeter functional may be seen
as a fractional counterpart of the mean curvature.
As $s\to1$, these notions approach the classical
objects in different senses (see e.g.~\cite{brezis, davila, WW,
CV, ADP, CV2}
for details). The limit as~$s\to0$ has also been
taken into account under various circumstances
(see e.g.~\cite{MAZ0, DPFV0}).

These fractional theories of geometric type found
very often concrete applications in real-world problems.
For instance, fractional perimeter functionals
naturally appear in the large-scale description of
interfaces of nonlocal phase transitions (see~\cite{SV-ga1, SV-ga2}).
A very natural application arises also in computer science:
indeed, the square pixels of small side~$\epsilon$
produce, along the diagonal, an error of order one for
the classical perimeter,
but an error of order only~$\epsilon^{1-s}$ for
the fractional perimeter. In this sense,
fractional objects are very useful to ``average out''
the errors caused by the possible fine anisotropic structure of the media.

Many results of great interest about the  fractional
mean curvature flow have been recently obtained
in~\cite{CMP1, CMP2, CMP3}.
See also~\cite{AV} for a detailed study of the fractional
mean curvature, with analogies and important differences
with respect to the classical case.
The question of the regularity
of the minimal surfaces corresponding to the
fractional perimeter has been investigated in~\cite{CRS, SV, CV2, Fi, BFV, JS},
several connections with the isoperimetric
problems have been studied in~\cite{Fu, FFMMM, I4}
and remarkable examples of surfaces of vanishing and constant fractional
mean curvature have been recently constructed in~\cite{Dav, CFSW}.

In this work we are interested in studying classical solutions to the $L^2$-gradient flow associated to the fractional perimeter. More precisely,  we consider a set $E_0$ and we are interested in a family $E_t$ that satisfies for every $x\in \partial E_t$ the law\footnote{As customary
in geometric flows, we use the notation~$x\in\partial E_t$
to denote the points of the evolving hypersurface.
Notice that, with a slight abuse of notation,
we also implicitly take~$x=x_t$ to be a function of~$t$,
due to the evolutionary nature of the problem.
Also, if we take~$\partial E_t$ to be parameterized
by a normal map of the sphere~$f:S^{n-1}\to[0,+\infty)$,
we also use the notation~$\partial E_t\ni x=f(\omega)\,\omega$.
Of course, since the surface is evolving in time,
this notation has to be read as~$x_t=f(\omega,t)\,\omega$,
but we will try to simplify notation whenever possible,
still keeping the arguments unambiguous.} of motion
\begin{equation} \partial_t x \cdot \nu=-H_s, \label{fmcfgeneral}\end{equation}
where $\nu$ is the outer normal to $E_t$ and the quantity $H_s$ is the fractional mean curvature 
defined by
\begin{equation}\label{Hs}
H_s(x,E):=\lim_{\delta\searrow0} s(1-s) \int_{\R^n\setminus B_\delta(x)}\frac{
\tilde\chi_E(y)}{|x-y|^{n+s}}\,dy.\end{equation}
Here above and in the sequel we use the notation
$$ \tilde\chi_E(y):=\chi_{\R^n\setminus E}(y)
-\chi_E(y),$$
while~$\chi_E$ is the classical indicator function of~$E$, that is~$1$ on~$E$
and~$0$ on~$\R^n\setminus E$.  We also assume that the parameter $s$
belongs to the interval~$(0,1)$.
Notice that with this convention the mean curvature of a sphere is positive
(more details on this case will be given in the forthcoming
Section~\ref{example}). Moreover, under this convention,
the $s$-perimeter of solutions to \eqref{fmcfgeneral} decreases in the fastest direction;
In fact it holds that (see Theorem \ref{derivatives non local quantities})
$$ \partial_t P_s(E_t)=
-\int_{\partial E_t} H_s^2(y)\,d{\mathcal{H}}^{n-1}(y)\leq 0.$$

The flow described by equation \eqref{fmcfgeneral} is the natural analog of the mean curvature flow, which has been studied largely in the literature (see for instance \cite{KEMCFBook}, \cite{EckerHuiskenAnn}, \cite{H}, \cite{HS}, \cite{M} , \cite{SS} and references therein). The mean curvature flow has been used in several contexts that range from  modeling  interface transition \cite{Mullins} to obtain topological classification of certain surfaces \cite{HS, BH}.

The mean curvature flow is a quasilinear geometric equation of parabolic nature that has regularizing effects as long as the mean curvature remains bounded (i.e. solutions are $C^\infty$ in space and time while the mean curvature is bounded), but it may form singularities in finite time. One of the main topics within the subject is the study of singularity formation  during the evolution. 
The 
first important result in this direction is due to G. Huisken \cite{H} who showed that convexity is preserved by the flow and that singularities only form at an extinction time at which the surface collapses to a ``round point'', that is after appropriate rescaling
 convex surfaces are asymptotic to spheres. Later on, it has been proved that in fact the flow preserves $k$-convexity for any $1\leq k\leq n-1$ (\cite{HS}) and that homothetic solutions play an important role in the understanding of singularity formation. 
In this paper we show that $H_s$-convexity is preserved by the fractional
flow (see Section \ref{for}) and we observe that in fact spheres are self-similar solutions to the flow (see Section \ref{example}).

Another important classical example of evolution by mean curvature flow is the evolution of entire graphs with linear growth. In \cite{EckerHuiskenAnn} it is shown that in that case the evolution exists and it is smooth for all times.  The estimates of that work were later 
localized in \cite{EckerHuiskenInvent} to obtain short time estimates for any evolution. Other graphical evolutions have been studied in \cite{SS}. In Section \ref{EGS} we show that  graphical solutions to \eqref{fmcfgeneral} have bounded $H_s$- curvature for all times and are in fact $C^\infty$. A key element of the proof is the preserved quantity $(\nu\cdot e_n)^{-1}$ which is known as the height function. 
On the other hand,  star-shaped surfaces also have a preserved quantity and we briefly address this case in Section \ref{ESS}.

Other results that we present here
are a comparison principle, the preservation of the
positivity of $H_s$ and some estimates for entire graphs.

The organization of the paper is as follows: Section \ref{sez:SC} is devoted to formulate Equation \eqref{fmcfgeneral} for star-shaped surfaces and entire graphs. We compute in particular the example of an evolving sphere. In Section \ref{comparison-section} we show that a comparison principle holds for the flow and as a corollary we find bounds on the maximal existence time and uniqueness of smooth solutions. Section \ref{EGQ} is devoted to compute the evolution of local and non-local geometric quantities. Of particular interest is the evolution equation of $H_s$ that is given by
$$\frac{\partial_t  H_s}{2 s(1-s)}(x)= 
\textrm{P.V.}\int_{\partial E_t}
 \frac{H_s(y)-H_s(x)
}{|x
-y|^{n+s}}dy+ H_s(x) \textrm{P.V.} \int_{\partial E_t}
 \frac{1-\nu(x)\cdot \nu(y)
}{|x
-y|^{n+s}}dy.$$
This equation implies that if the initial condition satisfies $H_s>0$ then this is preserved by the flow. This result is proved in Section \ref{for}. In Section \ref{EGS} we prove bounds for graphical solutions and in Section \ref{ESS} that star-shapedness is preserved by the flow as long as the fractional curvature remains bounded.

\section{Some special cases}\label{sez:SC}

In this section, we consider some particular forms of
the fractional mean curvature motion, namely the cases
in which the evolving surface is the boundary of
a star-shaped domain or it is a graph in a given direction.
A simple and concrete example of fractional mean curvature evolution
for star-shaped surfaces is given by the spheres,
in which the equation can be explicitly solved by scale invariance. On the other hand, planes are trivial examples of graphical evolutions.

\subsection{Evolution of star-shaped surfaces}

In this subsection we assume that the initial set is of the form
$$ E_0=\big\{ \rho\omega,\ \omega \in S^{n-1},\ \rho\in [0,f_0(\omega)]\big\}$$
with~$\nu(p)\cdot p\ge0$ for any~$p\in\partial E_0$, where~$\nu(p)$ is the outer
unit normal at~$p$.

We deal with the motion of~$\partial E_0$ by its fractional mean curvature.  We assume that this evolution is regular and star-shaped around the origin for all times $t\in [0,T)$
That is, we consider \begin{eqnarray*}
&&E_t=\big\{ \rho\omega,\ \omega \in S^{n-1},\ \rho\in [0,f(\omega,t)]\big\} \\
\hbox{ and } && \partial E_t=\big\{ f(\omega,t)\omega,\ \omega \in S^{n-1}\big\}\end{eqnarray*}
with~$f\in C^2\big(S^{n-1}\times(0,+\infty), [0,+\infty)\big)
\cap  C^0\big(S^{n-1}\times[0,+\infty), [0,+\infty)\big)$ and $f>0$.

In order to write 
~\eqref{fmcfgeneral} more explicitly
in dependence of~$f$ we extend the function~$f=f(\cdot,t)$,
that was originally defined on~$S^{n-1}$, to the whole of~$\R^n\setminus\{0\}$
by homogeneity, namely we suppose, without loss of generality, that~$
f:\R^n\setminus\{0\}\rightarrow[0,+\infty)$, with
\begin{equation}\label{f de}
f(x) = f\left( \frac{x}{|x|}\right)
\ {\mbox{ for every }} x\in\R^n\setminus\{0\}.\end{equation}
Notice that we omitted, for simplicity, the dependence
on the time~$t$ in the notation above. 
Similarly, given~$\omega\in S^{n-1}$, unless otherwise
specified, we denote by~$\nu$ the exterior normal at
the point~$f(\omega)\omega$. 
Hence we have:

\begin{lemma}\label{EXP}
The external normal $\nu$ of~$E$ can be expressed in terms of~$f$ by
\begin{equation}\label{NU}
\nu =\frac{f \omega-\nabla f}{\sqrt{|\nabla f|^2+f^2}}.\end{equation}
Also, given any~$\omega\in S^{n-1}$,
for any~$\eta\in\R^n$ orthogonal to~$\omega$ we have that
\begin{equation}\label{99}
(\nabla f(\omega)\cdot\eta)\,(\omega\cdot\nu)+f(\omega)\eta\cdot\nu=0.
\end{equation}
Finally,~\eqref{fmcfgeneral} is equivalent to
\begin{equation}\label{s-flow-in-f} \left\{
\begin{matrix}
& \partial_t f(\omega,t) = 
-H_s(*,E_t)\displaystyle\frac{\sqrt{|\nabla f|^2+f^2}}{f} ,\qquad
{\mbox{for every $\omega\in S^{n-1}$ and $t>0$,}}\\
\\
& f(\omega,0)=f_0(\omega),\qquad
{\mbox{for every $\omega\in S^{n-1}$,}}
\end{matrix}
\right.\end{equation}
where $*=f(\omega,t)\omega$.
\end{lemma}

\begin{proof} 
For the analogue of~\eqref{s-flow-in-f}
in the classical mean curvature flow see, e.g., formula~(2.8)
in~\cite{Ur}.
As a matter of fact,
the formula and its
proof are exactly the same as in the classical case, 
since here the specific choice of~$H$ or~$H_s$
(being the classical or nonlocal mean curvature)
does not play any role. We provide the details for the facility
of the reader.
For this, first we point out that, by~\eqref{f de},
\begin{equation}\label{000}
\nabla f(\omega)\cdot\omega = 
\left.\frac{d}{d\tau} f(\tau \omega)\right|_{\tau=1}
=\left.\frac{d}{d\tau} f(\omega)\right|_{\tau=1}
=0\end{equation}
for any~$\omega\in S^{n-1}$. Also, if~$\tau\mapsto \omega(\tau)$
is a curve on~$S^{n-1}$, we have that
\begin{equation}\label{00} \omega\cdot \dot\omega=\frac{d}{d\tau}\frac{|\omega|^2}{2}
=\frac{d}{d\tau}\frac{1}{2}=0\end{equation}
and a generic tangent vector at~$\partial E$ is
$$ T:=\frac{d}{d\tau} (f \omega) = (\nabla f\cdot\dot\omega)\omega
+f\dot\omega.$$
We observe that
$$ (f\omega-\nabla f)\cdot T=
f\,(\nabla f\cdot\dot\omega)
+f^2\dot\omega\cdot\omega-
(\nabla f\cdot\dot\omega)(\nabla f\cdot\omega)-
f\,(\nabla f\cdot \dot\omega)=0,$$
thanks to~\eqref{000} and~\eqref{00}.
This shows that the vector~$f\omega-\nabla f$
is normal to~$\partial E$. Also, by~\eqref{000}, the component of~$f\omega-\nabla f$
in direction~$\omega$ is~$f$, which is positive:
accordingly, this normal vector
points outwards and this completes the proof of~\eqref{NU}.

Using \eqref{000} and~\eqref{NU}, we also obtain that
\begin{equation}\label{pre99}
\omega\cdot\nu=
\frac{f}{\sqrt{|\nabla f|^2+f^2}},\end{equation}
and this shows that~\eqref{fmcfgeneral} and~\eqref{s-flow-in-f}
are equivalent (recall indeed that~$x=f(\omega)\omega$).

It remains to prove~\eqref{99}. For this,
we take~$\eta$ orthogonal to~$\omega$ and we use~\eqref{NU}
and~\eqref{pre99} to compute
\begin{eqnarray*}
&& (\nabla f\cdot\eta)\,(\omega\cdot\nu)+f\eta\cdot\nu
\\ &&\qquad =
\frac{f\,(\nabla f\cdot\eta)}{\sqrt{|\nabla f|^2+f^2}}
+\frac{f^2 \eta\cdot\omega-f(\eta\cdot\nabla f)}{\sqrt{|\nabla f|^2+f^2}}
\\ &&\qquad =
\frac{f\,(\nabla f\cdot\eta)}{\sqrt{|\nabla f|^2+f^2}}
+\frac{0-f(\eta\cdot\nabla f)}{\sqrt{|\nabla f|^2+f^2}}
\end{eqnarray*}
that clearly equals to zero and proves~\eqref{99}.
\end{proof}

\subsubsection{A concrete example: The evolution of spheres}\label{example}

In this section we compute the example
of a concrete evolution, namely we show that the
spheres shrink self-similarly in finite time.
We think it is a very interesting open problem
to determine whether or not these are
the only embedded self-similar shrinking solutions of~\eqref{fmcfgeneral}.

\begin{lemma}\label{SC}
We have, 
for any~$x\in\partial B_R(0)$,
\begin{equation}\label{EH2} 
H_s(x, B_R(0)) = \varpi R^{-s},\end{equation}
for some~$\varpi>0$.
\end{lemma}

\begin{proof} 
First we show that,
for any~$x\in\partial B_1(0)$,
\begin{equation}\label{EH1} 
H_s(x, B_1(0))=\varpi.\end{equation}
By rotational invariance of the integrals, we have that
$H_s(x_1,B_1(0))=H_s(x_2,B_1(0))$ for every $x_1,x_2\in \partial B_1(0)$,
thus showing~\eqref{EH1}.
Moreover, if~$\omega\in S^{n-1}$
and~$x=R\omega$,
by changing variable~$\tilde y:=Ry$, we see that
\begin{equation}\label{ca}\begin{split}
H_s(x,B_R(0))=&\lim_{\delta\searrow0} s(1-s) \int_{\R^n\setminus B_\delta(x)}\frac{
\tilde\chi_{B_R(0)}(\tilde y)}{|R\omega-\tilde y|^{n+s}}\,d\tilde y\\=& R^n
\lim_{\delta\searrow0} s(1-s) \int_{\R^n\setminus B_{R^{-1}\delta}(x)}\frac{
\tilde\chi_{B_R(0)}(Ry)}{|R\omega-Ry|^{n+s}}\,dy
\\=& R^{-s}
\lim_{\delta\searrow0} s(1-s) \int_{\R^n\setminus B_\delta(x)}\frac{
\tilde\chi_{B_1(0)}(y)}{|\omega-y|^{n+s}}\,dy
\\=& R^{-s} H_s(\omega,B_1(0))
.\end{split}\end{equation}
This, together with~\eqref{EH1}, proves~\eqref{EH2}.
\end{proof}

\begin{corollary}\label{COS}
Let~$\varpi$ be as in~\eqref{EH2} and~$C_0:=\varpi\,(s+1)$.
Let~$R(t):=(R_0^{s+1}-C_0 t)^{\frac{1}{s+1}}$.
Then~$B_{R(t)}(0)$ is a star-shaped
solution to fractional mean curvature flow with initial condition $B_{R_0}(0)$
and it collapses to the origin in the finite time~$\frac{R_0^{s+1}}{C_0}$.
\end{corollary}

\begin{proof} We only need to show that~\eqref{s-flow-in-f}
is satisfied with~$f(\omega,t):=R(t)$ and~$f_0(\omega):=R_0$.
For this, we use Lemma~\ref{SC} to compute
$$ \partial_t f+ H_s\,\frac{\sqrt{|\nabla f|^2+f^2}}{f} =
-\frac{C_0}{s+1} (R_0^{s+1}-C_0 t)^{\frac{-s}{s+1}} +H_s=
\varpi\,(R_0^{s+1}-C_0 t)^{\frac{-s}{s+1}}+\varpi\,R^{-s}=0,$$
that shows the validity of~\eqref{s-flow-in-f}.
\end{proof}

{F}rom the results in Section \ref{comparison-section}, we will
see that the one provided in Corollary~\ref{COS} 
is indeed the unique smooth solution
of the fractional mean curvature flow with spherical initial datum.

It is also easy to check that a similar computation yields an analogous result for the evolution of
cylinders.

\subsection{Evolution of graphical surfaces}
In this subsection we assume that the  initial set is of the form
$$ E_0=\big\{ (x,z), x \in \rr^{n-1}, z\in [-\infty, u(x)]\big\}.$$
The appropriate choice of normal in this situation is given by $$\nu(x,u(x))=\frac{(-\nabla u, 1)}{\sqrt{1+|\nabla u|^2}}.$$ 

We suppose that 
$$ E_t=\big\{ (x,z), x \in \rr^{n-1}, z\in (-\infty, u(x,t)]\big\} ,$$
with~$u\in C^2\big(\rr^{n-1}\times(0,+\infty), [0,+\infty)\big)\cap  
C^0\big(\rr^{n-1}\times[0,+\infty), [0,+\infty)\big)$.
In this setting,
we have that
$$ \partial E_t=\big\{ (x,u(x,t)), 
\, x\in \rr^{n-1}\big\}$$
and
the geometric flow in~\eqref{fmcfgeneral} is equivalent to
\begin{equation}\label{s-flow-graphical} \left\{
\begin{matrix}
& \partial_t u(x,t) = -H_s(x,E_t)\displaystyle \sqrt{|\nabla u|^2+1} ,\qquad
{\mbox{for every $x\in \rr^{n-1}$ and $t>0$,}}\\
\\
& u(x,0)=u_0(x),\qquad
{\mbox{for every $x\in \rr^{n-1}$.}}
\end{matrix}
\right.\end{equation}
A concrete example in this case is any linear $u$, which has fractional mean curvature equal to 0.

\begin{remark}
Equations \eqref{s-flow-in-f}  and \eqref{s-flow-graphical} are well posed imposing weaker regularity conditions on $f$ and $u$ respectively
\end{remark}

\section{Comparison principle}\label{comparison-section}

In this section we show that two surfaces evolving under fractional mean 
curvature flow that are initially nested remain nested while the 
evolution is smooth. 
To state the result, we will
consider two evolving surfaces, say~$E_t$ and~$F_t$,
and we use the notation
for points~$x(\cdot,t)\in \partial E_t$ and $y(\cdot,t)\in \partial F_t$. 
Then, we have the following comparison 
result:

\begin{theorem}\label{comparison general}
Let $E_t$ and $F_t$ be two smooth solutions to \eqref{fmcfgeneral} in $[0,
T)$ such that $E_0\subseteq F_0$. 
Assume additionally that $\partial_t x(\cdot,t)$,  $\partial_t 
y (\cdot,t)$ are continuous in $[0,T)$.
Then $E_t\subseteq F_t$.
\end{theorem}
\begin{proof}

We first assume that the closure of~$E_0$ is strictly contained in the interior of~$F_0$ and suppose that there is a time $t_0$ and a point $x_t
$ at which $E_{t_0}$ and  $\partial F_{t_0}$ touch for the first time and the normal velocity of $E_{t_0}$ at $x_t
$ is bigger than the normal velocity of $\partial F_{t_0}$ at that point (i.e. the boundaries cross at point of space time). Since $\partial E_{t_0}$ and  $\partial F_{t_0}$ are tangential at $x_t
$ the normal vectors agree at that point. Then we have
\begin{equation}\label{67:BA:01}
0\geq \left(\partial_t x_{F_t}-\partial_t x_{E_t}\right)\cdot \nu_{E}(x_t
)=H_s(E_t,x_t
)-H_s(F_t,x_t
)\end{equation}
Moreover,
we may suppose that the strict inclusion
\begin{equation}\label{67:BA:02}
E_{t_0}\subset F_{t_0}
\end{equation}
holds true (since if~$E_{t_0}=F_{t_0}$, the backward
flow would give~$E_{t}=F_{t}$ for all~$t\in[0,T)$).
{F}rom~\eqref{67:BA:02}, 
we have the strict inequality~$H_s(E_{t_0},x_t
)>H_s(F_{t_0},x_t
),$ which, inserted in~\eqref{67:BA:01}, yields a contradiction.

If the closure of~$E_0$ is not strictly contained in the interior of~$F_0$, then we can proceed as before by observing that the equation holds in the limit as $t\to 0$. \end{proof}

Theorem \ref{comparison general} implies uniqueness of smooth solutions to \eqref{fmcfgeneral}.

\begin{corollary}
There is at most
one smooth solution to \eqref{fmcfgeneral}.
\end{corollary}

\begin{proof}
Assume that $E_0=F_0$.
 By Theorem~\ref{comparison general}, we have that~$
F_t\subset E_t$ and $
E_t\subset F_t$.\end{proof}

By trapping the solution between balls, we obtain
estimates about the evolution of the fractional mean
curvature and the extinction time:

\begin{corollary}\label{T:S}
Let~$R>\delta>0$ and $E_t$ a solution to \eqref{fmcfgeneral}
such that there are $x_\delta$ and $x_R $  that satisfy $B_\delta(x_\delta)\subseteq E_0\subseteq B_R(x_R)$, then $B_{(\delta^{s+1}-C_0 t)^{\frac{1}{s+1}}}(x_\delta)\subseteq E_t\subset B_{(R^{s+1}-C_0 t)^{\frac{1}{s+1}}}(x_R)$.

In particular, if $f\in  C^1\big(S^{n-1}\times(0,T)\big)\cap  C^0\big(S^{n-1}\times[0,T]\big) $ is
a solution of~\eqref{s-flow-in-f}, with~$f(\omega,t)>0$
for every $(\omega,t)\in S^{n-1}\times[0,T]$,
 that satisfies
$\delta<f(\omega,0)<R$, for every~$\omega\in S^{n-1}$.
then 
\begin{equation}\label{ET-1}
(\delta^{s+1}-C_0 t)^{\frac{1}{s+1}}\leq f(\omega,t)\leq (R^{s+1}-C_0 t)^{\frac{1}{s+1}}.\end{equation}
Moreover,
the maximal existence time is bounded from above by $\frac{R^{s+1}}{C_0}$.
\end{corollary}

\begin{proof}
The result follows directly from Theorem \ref{comparison general}.
\end{proof}



\section{The evolution of the geometric quantities} \label{EGQ}

In this section we study the evolution of local and nonlocal geometric quantities.

We first remark that equation \eqref{fmcfgeneral}
is invariant under reparameterizations:  Suppose that $x$ satisfies \eqref{fmcfgeneral} and consider
a reparameterization $\varphi(\omega, t)$. Then we have that~$\tilde{x}=x(\varphi(\omega,t), t)$ satisfies
$$\partial_t\tilde{ x}\cdot \tilde{\nu}=(Dx(\partial_t \varphi  )+\partial_t x   )\cdot \tilde{\nu}=-H_s(\tilde{x}).$$
Moreover, by reparameterizing the smooth
surface with a time dependent parameter it is possible to obtain an evolution equation that has tangent velocity equal to 0.

\begin{theorem}\label{reparameterization}
Suppose that $E_t$ is smooth
and satisfies the evolution equation \eqref{fmcfgeneral}. Then, there is a parameterization of $\partial E_t$ such that
\begin{equation} \partial_t x(t)= -H_s(x(t), E_t)\,\nu,\label{simplified eq}\end{equation}
for $x\in \partial E_t$.
\end{theorem}

\begin{proof}
We follow the analogous proof for other geometric flows (see \cite{KEMCFBook} for instance). 
For this, as usual we denote
the metric of the evolving surface by~$g_{ij}$
and
the inverse of the metric by~$g^{ij}$.
Assume that $\partial E_t$ is parameterized by spatial coordinates $(\omega_1,\ldots, \omega_{n-1})\in U\subset \R^{n-1}$.
 Then we have that
$x(\omega,t)\in \partial E_t$  satisfies  \eqref{fmcfgeneral}. We want to reparameterize $\omega$ in term of new time-dependent  local coordinates.  Hence, we assume that the coordinates  $(\omega_1,\ldots, \omega_{n-1})$ are parameterized by a spatial parameter $\Theta=(\theta_1,\ldots, \theta_{n-1})$ and time $t$. Then we define
$$\Gamma(\Theta,t)=x(\omega(\Theta,t),t).$$
We have 
\begin{align*}\partial_t \Gamma=&\sum_i\partial_{\omega_i} x(\omega(\Theta,t),t)\partial_t\omega_i+\partial_t x (q,t)|_{q=\omega(\Theta,t)}\\
=&-H_s(\Gamma(\Theta,t)) \nu+\left. (\tau_i\partial_t\omega_i+(\partial_t x)^T)\right|_{q=\omega(\Theta,t)},
\end{align*}
where $\tau_i$ is the tangential vector $\partial_{\omega_i} x(\omega(\Theta,t),t)$ and $(\partial_tx)^T=\partial_t x-(\partial_t x\cdot \nu)\nu$ 
is the tangential part of $\partial_t x$.

Standard ODE theory implies the existence of a solution to  $$\partial_t\omega_i(\Theta, t)=(\partial_t x)^T g^{ij} \omega^T\cdot \tau_j,$$ with 
$\omega(\Theta, t_0)=\omega$ (the original parameterization at time $t_0$).

Hence, the surface $\Gamma(\Theta,t) $ satisfies \eqref{simplified eq} for time close to $t_0$. \end{proof}

In the next subsection we assume that $\Gamma(t)$ is the reparameterization of $\partial E_t$ described by 
Theorem \ref{reparameterization}. For simplicity, we still denote the spatial parameter as $\omega\in U \subset  \rr^{n-1}$ or $x\in \rr^{n-1}$.

\subsection{Evolution of local quantities}
In this subsection we consider the evolution of some geometric quantities associated to $\partial E_t$. We assume that~$\partial E_t$ is smooth. 

Consider $\Gamma(t)$ satisfying \eqref{simplified eq}.  
We start by recalling the definition of the metric $g_{ij}$, the second 
fundamental form $a_{ij}$ and the square of  its norm~$|A|^2$. Here we denote by $ 
(m_{ij})$ the matrix of components $m_{ij}$ and we use Einstein's 
summation convention whenever repeated indices occur. We denote the 
inverse of the metric as $g^{ij}$ and we raise indices of matrices to 
indicate contraction by this matrix (e.g. $m^i_j=g^{ij}m_{ij}$).
In this setting, we have:
\begin{equation}\label{PRE}
\begin{split}g_{ij}=\,&\partial_{\omega_i} \Gamma \cdot \partial_{\omega_j}\Gamma,\\
 (g^{ij})=\,&(g_{ij})^{-1},\\
 a_{ij}=\,&\partial_{\omega_i}\nu\cdot\partial_{\omega_j} \Gamma=-\nu \cdot \partial_{\omega_j}\partial_{\omega_j} \Gamma=\partial_{\omega_j}\nu\cdot\partial_{\omega_i} \Gamma,\\
 |A|^2=\,&g^{ij}a_{ik} g^{kl}a_{jl}.
\end{split}\end{equation}

We also denote $$\nabla^{\Gamma} F=g^{ij}\partial_{\omega_j} F \partial_{\omega_i}\Gamma,$$ which correspond to projecting the gradient of $F$ on the tangent space (for a globally defined function) and 
 $$\nabla^{\Gamma}_i X^j=\partial_{\omega_i} X^j+C^j_{ik}X^k,$$ where $C^j_{ik}$ are the Christoffel symbols on the surface. 
 
\begin{theorem}\label{local quantities}
Assume that $\Gamma(\Theta,t)=\partial E_t$ is parameterized such that it satisfies  \eqref{simplified eq}.
Then we have that
\begin{eqnarray}
\label{evolutionmetric}  \partial_t g_{ij}=&-2H_s a_{ij},
\\  \label{evolutioninverse}   \partial_t g^{ij}=&2H_s a^{ij},
\\ \label{evolutionnormal}  \partial_t \nu=&\nabla^{\Gamma} H_s,
\\  \label{evolutioniifundform} \partial_t a_{ij}
=&\nabla^\Gamma_i\nabla^\Gamma_j H_s-H_s a_{ik}a^{k}_j
\\  \label{evolutionnorm2ndfundform} \partial_t 
|A|^2=&2a^{ij}\nabla^\Gamma_i\nabla^\Gamma_j H_s+2H_sa_{ik}a^k_ja^{ij}.
\end{eqnarray}
\end{theorem}

\begin{proof}
The proofs are similar to the local case (see \cite{KEMCFBook} for instance).
First, we prove~\eqref{evolutionmetric}
by computing the evolution of the metric: we recall
that~$\partial_{\omega_i}\Gamma$ is a tangent vector, thus \begin{equation}\label{QUESTA}
\partial_{\omega_i}\Gamma\cdot\nu=0.\end{equation} Also~$\Gamma$ satisfies~\eqref{simplified eq},
and so~$\partial_t \Gamma=-H_s\nu$. As a consequence,
\begin{align*}\partial_t g_{ij}=\,
&\partial_{\omega_i}(\partial_t \Gamma)\cdot  \partial_{\omega_j}\Gamma
+ \partial_{\omega_i}\Gamma\cdot  \partial_{\omega_j}(\partial_t\Gamma)\\
=\,&\partial_{\omega_i}(-H_s \nu)\cdot  \partial_{\omega_j}\Gamma
+ \partial_{\omega_i}\Gamma\cdot  \partial_{\omega_j}(-H_s \nu)\\
=\,&-2H_sa_{ij} \end{align*}
and so
we obtain \eqref{evolutionmetric}.

Now, since $g_{ij}\,g^{jk}=\delta_{i}^{k}$ (here we are adding on the repeated index $j$), using \eqref{evolutionmetric}
we have that
$$ 0 =\partial_t\delta_{i}^{k}= \partial_t g_{ij}\, g^{jk}+
g_{ij}\, \partial_t g^{jk}= -2H_s\,a_{ij}g^{jk}+
g_{ij}\, \partial_t g^{jk},$$
which gives \eqref{evolutioninverse}.

Also, using that~$\nu\cdot \nu =1$ and \eqref{QUESTA},
we have that
$$ \partial_t \nu \cdot \nu= 0,$$
that
$$ \partial_{\omega_i} \nu \cdot \nu= 0$$
and
$$ \partial_t \nu \cdot \partial_{\omega_i}\Gamma=
-\nu \cdot  \partial_{\omega_i}(\partial_t \Gamma)
=\nu\cdot \partial_{\omega_i}(H_s  \nu)=
\partial_{\omega_i}H_s.$$
Hence, decomposing $\partial_t\nu$ along the orthogonal directions~$\{\nu,
\partial_{\omega_1} \Gamma,\dots, \partial_{\omega_{n-1}} \Gamma\}$, we conclude that
$$\partial_t \nu=g^{ij}\partial_{\omega_j} H_s \partial_{\omega_i} \Gamma=\nabla^\Gamma H_s.$$
This completes the proof of \eqref{evolutionnormal}.

Now we use \eqref{PRE} and \eqref{evolutionnormal} and we obtain that
\begin{align*}\partial_t a_{ij}= \,&- \partial_t
\nu\cdot\partial_{\omega_j}\partial_{\omega_j} \Gamma+ \nu\cdot\partial_{\omega_j}\partial_{\omega_j}  (H_s\nu)\\
=\,&-\nabla^{\Gamma} H_s \cdot\partial_{\omega_j}\partial_{\omega_j} \Gamma
+ \partial_{\omega_j}\partial_{\omega_j}  H_s
+H_s  \nu\cdot\partial_{\omega_j}\partial_{\omega_j}  \nu.
\end{align*}
Moreover,
\begin{eqnarray*}
&& 0=\frac12 \partial_{\omega_j}\partial_{\omega_j}(\nu\cdot\nu)
=\partial_{\omega_j}(\nu\cdot\partial_{\omega_j}\nu)=
\nu\cdot\partial_{\omega_j} \partial_{\omega_j}\nu
+\partial_{\omega_j}\nu\cdot\partial_{\omega_j}\nu
\end{eqnarray*}
and so we see that
\begin{equation}\label{PKL09}
\partial_t a_{ij}= 
-\nabla^{\Gamma} H_s \cdot\partial_{\omega_j}\partial_{\omega_j} \Gamma
+ \partial_{\omega_j}\partial_{\omega_j}  H_s
-H_s  \partial_{\omega_j}\nu\cdot\partial_{\omega_j} \nu.\end{equation}
Now we assume that we have normal coordinates at $x_t
$. Then at $x_t
$ the metric $g_{ij}$ equals to~$\delta_{ij}$ and the Christoffel symbols are 0. In particular, formula~\eqref{PKL09} reduces to 
\begin{align*}\partial_t a_{ij}= & \partial_{\omega_j}\partial_{\omega_j}  H_s-H_s\partial_{\omega_i}\nu\cdot \partial_{\omega_j} \nu\\
=& \partial_{\omega_j}\partial_{\omega_j}  H_s-H_sa_{ik}a^k_j.\end{align*}
Since in normal coordinates $ \partial_{\omega_j}\partial_{\omega_j}  H_s= \nabla^\Gamma_i\nabla^\Gamma_j H_s$ and the latter is a coordinate invariant quantity, this establishes~\eqref{evolutioniifundform}.

Now we prove~\eqref{evolutionnorm2ndfundform}. For this,
we use~\eqref{evolutioninverse}
and~\eqref{evolutioniifundform}, and we see that
\begin{eqnarray*}
\partial_t (g^{ij}a_{ik})&=&\partial_t g^{ij}a_{ik}+
g^{ij}\partial_t a_{ik}\\
&=&
2H_s a^{ij}a_{ik}+g^{ij} (\nabla_i^\Gamma\nabla_k^\Gamma H_s-
H_s a_{im}a^m_k)\\&=&
2H_s a^{ij}a_{ik}+g^{ij} \nabla_i^\Gamma\nabla_k^\Gamma H_s-H_s a^j_{m}a^m_k.
\end{eqnarray*}
Therefore
\begin{eqnarray*}
\partial_t (g^{ij}a_{ik})\,(g^{kl}a_{jl})&=&
2H_s a^{ij}a_{ik}g^{kl}a_{jl}
+g^{ij}g^{kl}a_{jl}\nabla_i^\Gamma\nabla_k^\Gamma H_s-
H_s a^j_{m}a^m_k g^{kl}a_{jl}
\\&=&
2H_s a^{ij}a_{i}^l a_{jl}
+a^{ik}\nabla_i^\Gamma\nabla_k^\Gamma H_s-
H_s a^j_{m} a^{ml} a_{jl}\\
&=&
H_s a^{ij}a_{i}^l a_{jl}
+a^{ik}\nabla_i^\Gamma\nabla_k^\Gamma H_s.\end{eqnarray*}
This and the fact that (recall~\eqref{PRE})
\begin{eqnarray*}
\partial_t |A|^2&=& \partial_t (g^{ij}a_{ik} g^{kl}a_{jl})\\
&=&\partial_t (g^{ij}a_{ik}) g^{kl}a_{jl}+
\partial_t (g^{kl}a_{jl}) g^{ij}a_{ik}\\
&=&2 \partial_t (g^{ij}a_{ik}) g^{kl}a_{jl}
\end{eqnarray*}
imply \eqref{evolutionnorm2ndfundform}.
\end{proof}

For further reference,
we also point out the following computation in local coordinates:

\begin{lemma}\label{evolution of tangents}
For local coordinates $\{\omega_1,\ldots, \omega_{n-1}\}$ we have that
$$\partial_t \left(\partial_{\omega_i} \Gamma\right)=-H_s 
a^j_i\partial_{\omega_j}\Gamma-\partial_{\omega_i}H_s \nu.$$
\end{lemma}

\begin{proof} Since~$\Gamma$ satisfies~\eqref{simplified eq},
$$ \partial_t(\partial_{\omega_i}\Gamma)=
\partial_{\omega_i}(\partial_{t}\Gamma)=\partial_{\omega_i}(-H_s\nu)=
-\partial_{\omega_i}H_s\nu-H_s\partial_{\omega_i}\nu.$$
On the other hand, by definition
$$ \partial_{\omega_i}\nu=a^j_i\partial_{\omega_j}\Gamma,$$
which implies the result.
\end{proof}

\subsection{Evolution of non-local quantities}

In this subsection we will analyze the evolution of the perimeter, the fractional mean curvature and their
first order  spatial derivatives. In order to simplify the notation we write the point $x(t) \in \partial E_t$ and the unit normal vector $\nu(x(t))$  to $\partial E_t$ at $x(t)$ as
$$x_t:=x(t) \hbox{ and } \nu_t:=\nu(x(t)).$$ 
We remark that 
when we integrate on the surface   $\partial E_t$ the
integration variable, that we usually denote by $y$,  depends on $t$, but we do not make explicit this dependence. Note additionally that $v\cdot w$ denotes the standard dot product on $\rr^n$ between the vectors $v$ and $w$.


We observe that the integrand in~\eqref{Hs} carries a singular kernel, therefore
it is convenient to remove such singularity by using a cancellation.
We perform these computations here, and we will use them
in the forthcoming Section~\ref{for} to show that the positivity
of the fractional mean curvature is preserved by the geometric flow.

To this goal, we write 
$$H_s(x_t,E_t)= H_s^{\textrm{reg}}(x_t,E_t)+H_s^{\textrm{sing}}(x_t
,E_t), $$
with
\begin{equation} H_s^{\textrm{sing}}(x_t
,E_t)= \lim_{\delta\searrow0}  s(1-s) \int_{C_R^{\nu_t}(x_t
)\setminus B_\delta(x_t
)}\frac{
\tilde\chi_{E_t}(y)}{|x_t
-y|^{n+s}}\,dy, \hbox{ and }\label{Hsing}\end{equation}
\begin{equation} H_s^{\textrm{reg}}(x_t
,E_t)= s(1-s) \int_{\R^n\setminus C_R^{\nu_t}(x_t
)}\frac{
\tilde\chi_{E_t}(y)}{|x_t
-y|^{n+s}}\,dy,\label{Hreg}\end{equation}
where $C_R^{\nu_t}(x_t
)$ is a fixed
cylinder centered at $x_t
$ with flat direction parallel to the normal of the surface at $x_t
$, namely
$$ C_R^{\nu_t}(x_t
):=\Big\{
x\in\rr^n {\mbox{ s.t. }}x=x_t
+y {\mbox{ with }}
|y\cdot\nu(x_t
)|<R {\mbox{ and }}
|y-(y\cdot\nu(x_t
))\nu(x_t
)|<R
\Big\}.$$
In what follows, we denote the surface $\partial E_t$ as $\Gamma(\omega,t)$ and we assume that is parameterized such that \eqref{simplified eq} holds.
Consider $x_t
\in \Gamma$ and the epigraph of the tangent plane $\Pi$ at $x_t
$ given by

\begin{equation} \Pi(x_t
,E_t):= \{ \xi\in\R^n {\mbox{ s.t. }} \nu_t \cdot (\xi-x_t
)\ge0\},\label{defPi}\end{equation}
where $\nu_t$ is the unit normal to $\Gamma(t)$ at the point $x_t
$.

Note that for $R$ small enough, $\Gamma(t)$  can be written as a graph over the tangent plane at $x_t
\in \Gamma(t)$. More precisely, 
let $\nu_t$ be the normal vector at $x_t
$ and let us parameterize $\partial \Pi$ (or equivalently, the linear space perpendicular to $\nu_t$) in appropriate polar coordinates $(r, \varphi)\in [0,R]\times S^{n-2}$. Then 
using the implicit function theorem, near $x_t$ we may define
a function 
$h$ such that
\begin{equation}\Gamma (\omega,t)=x_t
+\rho M_{x_t
} \varphi+\rho h(\rho, \varphi)\nu_t.\label{defg}\end{equation}
Here  $\rho$ is the distance to $x_t
$ on  $\partial \Pi$.
Also, we are implicitly identifying~$\partial\Pi$ with~$\R^{n-1}$
and
embedding $(n-1)$-dimensional
spaces into~$n$-dimensional spaces,
namely~$M_{x_t}$ is an $(n-1)$-dimensional matrix
acting on a vector~$\varphi\in S^{n-2}\subset\R^{n-1}$:
then the product~$M_{x_t
}\varphi$, as an $(n-1)$-dimensional vector, has to be thought to lie
on~$\partial\Pi$, which in turn is naturally embedded in the ambient space~$\R^n$.
More precisely, $M_{x_t
}\varphi \in \partial \Pi$ is defined as follows: 

\medskip

 Assume that $x_t
=\Gamma(\bar{\omega},t)$. 
Consider an orthonormal frame $\{v_j\}$ on $ \partial \Pi(x_t
,E_t)$. Since $\{\partial_{\omega_i}\Gamma\}_{\{i=1,\ldots n-1\}}$ span 
 $ \partial \Pi(x_t
,E_t)$, there are $c^{ij}(t_0)$ that satisfy $$v_j=c^{ji}(t_0)\partial_{\omega_i}\Gamma.$$
 We define  $c^{ji}(t)$  for $t\leq t_0$  
 as solutions to the ODE system
\begin{align}\partial_t c^{ij}-c^{rj}a^{i}_r(\bar{\omega},t) H_s(\Gamma(\bar{\omega},t))&=0 \label{def c}\\c^{ji}(t)|_{t=t_0}=c^{ji}(t_0).\end{align}
Notice that, for technical convenience, we are taking here the backward ODE flow from time~$t_0$.
 Then for $t\leq t_0$ we define  
 \begin{equation}v_j(\bar{\omega},t)=c^{ji}(t) \partial_{\omega_i}\Gamma(\bar{\omega},t).\label{def v} \end{equation}
We note
that~$v_j(\bar{\omega},t_0)=v_j$ and $\{v_j(t)\}\subset \partial \Pi(x_t
, E_t)$, where $x_t
=\Gamma(\bar{\omega},t)$ and $ \partial \Pi(x_t
, E_t)$ is the tangent plane of $\Gamma(\bar{\omega},t)$.

From \eqref{def c} and Lemma  \ref{evolution of tangents} 
\begin{align}\partial_t v_j& =\partial_t c^{ji}(t) \partial_{\omega_i}\Gamma(\bar{\omega},t)+c^{ji}(t) \partial_t(\partial_{\omega_i}\Gamma(\bar{\omega},t)
)\notag \\ & =- ( \nabla^{\Gamma}H_s\cdot  v_j)\,  \nu_t.\label{evolution of v}\end{align}
 Moreover, 
$$ \partial_t  (v_j\cdot v_i)=- (\nabla^{\Gamma}H_s \cdot v_j )   (\nu_t \cdot v_i) - (\nabla^{\Gamma}H_s\cdot v_i )  ( v_j\cdot \nu) =0.$$ 
Hence, $\{v_j\}$ remains an orthonormal base of $\Pi(x_t
, E_t)$.

Now we define 
\begin{equation} M_{x_t
} \varphi=\varphi^iv_i,\hbox{  where }\varphi \in S^{n-2}.\label{defM}\end{equation}

In particular, if we denote $x_t
=\Gamma(\bar{\omega},t)$, from \eqref{evolution of v} we have
\begin{equation}
\partial_t  M_{x_t
} \varphi= -( \nabla^{\Gamma}H_s\cdot M_{x_t
} \varphi) \, \nu_t \label{der in t of M}.\end{equation}

We also  note that, from equation \eqref{defg} and the quadratic
separation of the smooth surfaces from their tangent planes, it follows that $h(0, \varphi)=0$.

Notice also  that by symmetry, for  $\Pi_t:=\Pi(x_t
,E_t)$ and any $R>\delta>0$
\begin{equation} \int_{C_R^{\nu_t}(x_t
)\setminus B_\delta(x_t
)}\frac{
\tilde\chi_{\Pi_t
} (y)}{|x_t
-y|^{n+s}}\,dy=0.\label{plane integral}\end{equation}

Then, parameterizing $C_R^{\nu_t}(x_t
)$ as $x_t
+\rho M_{x_t
} \varphi+\rho z\nu_t$ with $\rho\in[0,R]$,
$\varphi\in S^{n-2}$ and $z\in[-R,R]$, due to cancellations we have that
\begin{equation}\begin{split} &H_s^{\textrm{sing}}(x_t
,E)= \lim_{\delta\searrow0}  s(1-s) \int_{C_R^{\nu_t}(x_t
)\setminus B_\delta(x_t
)}\frac{
\tilde\chi_{E_t}(y)+\tilde\chi_{\Pi_t} (y)}{|x_t
-y|^{n+s}}\,dy\\&\qquad
=s(1-s)\int_{S^{n-2}} 
\left[ \int_0^R \rho^{-1-s} \left(
\int_{h(\rho,\varphi)}^0\frac{1}{(z^2+1)^{\frac{n+s}{2}}}dz\right)
\, d\rho\right]\, d\varphi, \label{eqhsing}\end{split}\end{equation}
where $\Pi_t= \Pi_t(x_t
,E_t)$. We now compute the derivatives of $h$.

\begin{proposition}\label{derivatives of h}

For a given time $t$, consider a point $x_t
=\Gamma(\bar{\omega},t)$ and $\nu_t$ the normal vector to $\Gamma$ at $x_t$.  
Let $h$ be given by \eqref{defg} where $x_t$ is fixed as above. Then denoting by $\nu$ the normal to $\Gamma(\omega,t)$, we have that

 \begin{align*} \partial_t h(\rho, \varphi)=&\frac{1}{\rho}
\Big( 
H_s(x_t
)-H_s(\Gamma) \nu\cdot\nu_t\Big) +( \nabla^\Gamma  H_s(x_t
)\cdot  M_{x_t
}\varphi)
+\frac{1}{\rho}(\nu_t\cdot  D_{\omega}\Gamma (\omega,t) \partial_t \omega)   
,\\
  \partial_{\bar{\omega}_i} h(\rho, \varphi)=&\frac{\nu_t\cdot  D_{\omega}\Gamma (\omega,t)\partial_{\bar{\omega}_i}\omega +A(M_{x_t
}\varphi,\partial_{\bar{\omega_i}}\Gamma)  }{\rho},
 \end{align*}
where $A$ denotes the second fundamental form of $\Gamma(t)$ at $x_t$ and 
\begin{eqnarray*}
\partial_t{\omega_j}&=&\Big( (g^{ij}(x_t
)+O(\rho)\Big)\,\Big(
H_s(\Gamma) (D_{\omega}\Gamma (\bar{\omega}, t))^T\nu -\rho h(\rho, \varphi)(D_{\omega}\Gamma (\bar{\omega}, t))^T\nabla^\Gamma H(x_t
)\Big)
\\&\sim & H_s(\Gamma)\Big( O(\rho)+ O(\rho^2)\Big).\end{eqnarray*}
\end{proposition}
 \begin{proof}
 
 First, we note that from \eqref{defg}, $\omega$ becomes implicitly a function of $\varphi$ and $\rho$, but also of $x_t
$, hence it does depend implicitly on $t$. Hence,
taking derivatives on equation \eqref{defg} we have
 \begin{align}  D_{\omega}\Gamma (\omega,t)\partial_t{\omega}+\partial_t\Gamma =&\partial_t x_t
+\rho \partial_tM_{x_t
} \varphi+
 \rho \partial_t h(\rho, \varphi)\nu_t +\rho h(\rho, \varphi)\partial_t \nu_t.\label{dert} \end{align}
 
Note that
\begin{equation}\label{LKy67:1}
\partial_t \Gamma\cdot\nu_t =- H_s(\Gamma)\,\nu\cdot\nu_t\qquad{\mbox{ and }}\qquad
\partial_t x_t
\cdot\nu_t=-H_s(x_t
)\,\nu_t\cdot\nu_t=
-H_s(x_t
).\end{equation}
Moreover, since~$M_{x_t
}\varphi$ is a tangential vector at $x_t
$,
we have that~$M_{x_t
}\varphi\cdot\nu_t=0$, thus
\begin{equation}\label{LKy67:2}
-\partial_t M_{x_t
}\varphi\cdot\nu_t=
M_{x_t
}\varphi\cdot \partial_t \nu_t=
M_{x_t
}\varphi \cdot \,\nabla^\Gamma H_s(x_t
),\end{equation}
where the latter identity follows from~\eqref{evolutionnormal}.
Then, 
 using \eqref{simplified eq} and taking dot product with $\nu_t$ (recall also that~$\partial_t\nu_t\cdot\nu_t$),
we have 
  \begin{align*} \partial_t h(\rho, \varphi)=&\frac{1}{\rho}\Big(
H_s(x_t
)-H_s(\Gamma) \nu\cdot\nu_t\Big) + \nabla^\Gamma H_s(x_t
) \cdot  
M_{x_t
}\varphi
+\frac{1}{\rho}
 \nu_t \cdot D_{\omega}\Gamma (\omega,t) \partial_t \omega 
.  \end{align*}
  Now we are left to compute $\partial_t \omega $. To this end, 
we multiply equation \eqref{dert}  by~$D_{\omega}\Gamma (\bar{\omega}, t))^T$,
we exploit~\eqref{LKy67:1} and~\eqref{LKy67:2}
and we obtain 
  $$(D_{\omega}\Gamma (\bar{\omega}, t))^T D_{\omega}\Gamma (\omega,t))\partial_t{\omega}-H_s(\Gamma) D_{\omega}\Gamma (\bar{\omega}, t))^T\nu = \rho h(\rho, \varphi)(D_{\omega}\Gamma (\bar{\omega}, t))^T\nabla^\Gamma H(x_t
).$$
  Since $(D_{\omega}\Gamma (\bar{\omega}, t))^T D_{\omega}\Gamma (\bar{\omega},t)=(g_{ij}(x_t
))$, we have that the first matrix is $(g_{ij}(x_t
)+ O(\rho))$. Similarly, since 
  $D_{\omega}\Gamma (\bar{\omega}, t))^T\nu_t=0$, the second term is like $H_s(\Gamma) O(\rho)$. Hence
  \begin{align*}
  \partial_t{\omega}=&\Big(g^{ij}(x_t
)+O(\rho)\Big)\,\Big(
H_s(\Gamma) (D_{\omega}\Gamma (\bar{\omega}, t))^T\nu +\rho h(\rho, \varphi)(D_{\omega}\Gamma (\bar{\omega}, t))^T\nabla^\Gamma H(x_t
)\Big)\\
  \sim& \, H_s(\Gamma)\Big( O(\rho)+ O(\rho^2)\Big),
  \end{align*}
as desired.\end{proof}
We will also use 
a rotation that aligns the cylinder $C_R^{\nu_t}(x_t)$ with $C_R^{\nu_\tau}(x_\tau)$. 
We remark that since the vectors $\{v_i(t):\, i\ldots n-1\} \cup \{\nu_t\}$ are an orthonormal basis of $\rr^n$ we may  define for $y=y^i v_i(t)+y^n \nu_t$ the following rotation
\begin{equation}
\calR_{t,\tau} y=y^i v_i(\tau)+y^n \nu_\tau.\label{rotation}\end{equation}
Notice that 
\begin{equation}\label{DEsR97}
y^i=y\cdot v_i(t)\quad{\mbox{ and }}\quad
y^n=y\cdot \nu_t.\end{equation}
Then it is direct to show that
\begin{proposition} \label{derivative of rotation}
Consider $\calR_{t,\tau}$ given by \eqref{rotation} and denote  $\nabla^{\Gamma} H_s(\tau)$ the tangential  gradient of $H_s(x_\tau)$. Then it holds that
\begin{enumerate}
\item $\calR_{\tau,\tau}=\textrm{Id}$.
\item $\partial_{\tau_2} \calR_{\tau_1,\tau_2} y=\left.\left[ (y\cdot v_i(\tau_1)) \partial_t v_i(t)+(y\cdot \nu_{\tau_1}) \partial_t \nu_t\right)\right|_{t=\tau_2}=-(y\cdot v_i(\tau_1) )( v_i(\tau_2)\cdot \nabla^{\Gamma} H_s(\tau_2)) \,\nu_{\tau_2}+(y\cdot \nu_{\tau_1})  \nabla^{\Gamma} H_s(\tau_2) $.
\end{enumerate}
\end{proposition}

\begin{proof} For the sake of completeness, we give a proof of the second claim.
Using~\eqref{rotation} and~\eqref{DEsR97}, we have that
\begin{equation}\label{78:rtdsjky36g82}
\begin{split}
\partial_{\tau}
\calR_{t,\tau} y\,
& =y^i \partial_{\tau}v_i(\tau)+y^n \partial_{\tau}\nu_\tau\\
& =(y\cdot v_i(t)) \partial_{\tau}v_i(\tau)+(y\cdot \nu_t) 
\partial_{\tau}\nu_\tau
,\end{split}\end{equation}
which is one of the desired results.
In addition, from~\eqref{evolutionnormal}
we know that~$\partial_\tau \nu_\tau=\nabla^{\Gamma} H_s(\tau)$
and from~\eqref{evolution of v} that
$\partial_\tau v_i(\tau)
=- ( \nabla^{\Gamma}H_s(\tau)\cdot  v_i(\tau))\,  \nu_\tau$.
Hence we insert these pieces of information in~\eqref{78:rtdsjky36g82}
and we conclude that
$$ \partial_{\tau}
\calR_{t,\tau} y
=-(y\cdot v_i(t)) 
(v_i(\tau)\cdot \nabla^{\Gamma}H_s(\tau))\,  \nu_\tau
+(y\cdot \nu_t) \nabla^{\Gamma} H_s(\tau),$$
as desired.
\end{proof}

Now we study the evolution of the $s$-perimeter~$P_s$
and of the $s$-mean curvature.

\begin{theorem}\label{derivatives non local quantities}
Let~$x=x_t$, $\nu=\nu_t$  and $h$ be as in~\eqref{defg}.
We have the following equations:
\begin{eqnarray}
\label{F1} && \partial_t P_s(E_t)=
-\int_{\partial E_t} H_s^2(y)\,d{\mathcal{H}}^{n-1}(y)
\leq 0,
\\ \label{F4} && \frac{\nabla^{\Gamma}_{i} H_s}{s(1-s)}(x)= (n+s )
g^{ij}\left( \textrm{P.V.}\int_{\R^n}\tilde\chi_{E_t} \,\frac{(y-x)\cdot \partial_{\omega_j} x}{
|x-y|^{n+s+2}}dx\right),
\\ 
\label{F5} {\mbox{and }}&&\frac{ \partial_t H_s}{2s(1-s)}(x)
=\textrm{ P.V.} \int_{\partial E_t} \frac{(\partial _t x
-\partial_t y)\cdot\nu(y)}{|x
-y|^{n+s}}dy\\
&&\qquad=
\textrm{P.V.}\int_{\partial E_t}
 \frac{H_s(y)-H_s(x)
}{|x
-y|^{n+s}}dy+ H_s(x) \textrm{P.V.} \int_{\partial E_t}
 \frac{1-\nu(x)\cdot \nu(y)
}{|x
-y|^{n+s}}dy.\nonumber
\end{eqnarray}

Also,
\begin{eqnarray}
\label{F6} && {\mbox{the function }}\;
(0,R)\times S^{n-2}\ni(\rho,\varphi)\mapsto
\frac{ 
\rho^{-1-s}\,\partial_t h(\rho,\varphi)}{
(1+h^2(\rho,\varphi))^{\frac{n+s}{2}} }
{\mbox{ is integrable, }}\\
&& \qquad \partial_t( H_s^{\textrm{sing}})=O(R^{1-s}) \hbox{ and }\nonumber \\
&&\int_{S^{n-1}}\int_{-1}^1( \chi_{E_t}+\chi_{\Pi_t})\big( x_t+RM_{x_t}\omega +Rz\nu_t(x_t)\big)
\frac{z M_{x_t}\omega\cdot\nabla^\Gamma H_s(x_t)}{(1+z^2)^{\frac{n+s}{2}}} dz d\omega=O(R). \label{limit on the cylinder}
\end{eqnarray} 
\end{theorem}

\begin{proof}
In the course of the proof, we will also establish the auxiliary formulas
\begin{eqnarray}
\label{F2} &&\frac{ \partial_t( H_s^{\textrm{sing}})(x_t)}{s(1-s)}
= -\int_{S^{n-2}}\left[ \int_0^R 
\frac{ 
\rho^{-1-s}\,\partial_t h(\rho,\varphi)
}{ (1+h^2(\rho,\varphi))^{ \frac{n+s}{2} } }
\; d\rho
\right] 
\; d\varphi,\\
\label{F3} && \frac{\partial_t( H_s^{\textrm{reg}})(x_t)}{s(1-s)}= 2\int_{(\partial E_t)\setminus C_R^{\nu_t}(x_t
)}
 \frac{(\partial _t x_t
-\partial_t y
+ (y-x_t)\cdot \nabla^\Gamma H_s \, \nu_t
-(y-x)\cdot \nu_t \,
\nabla^\Gamma H_s(x_t)) \cdot\nu}{|x_t
-y|^{n+s}}dy
\\
\notag &&= 2\int_{(\partial E_t)\setminus C_R^{\nu_t}(x_t
)}
 \frac{(\partial _t x_t
-\partial_t y )\cdot\nu}{|x_t
-y|^{n+s}}dy\\
\notag &&\qquad +R^{-s}\int_{S^{n-1}}\int_{-1}^1( \chi_{E_t}+\chi_{\Pi_t})\big( x_t+RM_{x_t}\omega +Rz\nu_t(x_t)\big)
\frac{z M_{x_t}\omega\cdot\nabla^\Gamma H_s(x_t)}{(1+z^2)^{\frac{n+s}{2}}} dz d\omega,
 \\
 \notag  && \mbox{ where $\nu$ is the unit normal vector to } \partial E_t \hbox{ at } y \hbox{ and }\Pi_t \hbox{ is defined as in \eqref{defPi}}
\end{eqnarray}

Formula~\eqref{F1} follows from Theorem 6.1 in \cite{FFMMM}
and~\eqref{F2} from \eqref{eqhsing}.

To compute the derivative of the regular part
we need to compute
$$\lim_{h\to 0}  \frac{H_s^{\textrm{reg}}(x_t
(t), E_{t})-H_s^{\textrm{reg}}(x_t
(t-h), E_{t-h})}{h}=\frac{1}{h}\left(\int_{\rr^n\setminus C_R^{\nu_t}(x_t)}\frac{
\tilde{\chi}_{E_t}(y)}{|x_{t}-y|^{n+s}}
- \int_{\R^n\setminus C_R^{\nu_{t-h}}(x_{t-h})}\frac{
\tilde{\chi}_{E_{t-h}}(y)}{|x_{t-h}-y|^{n+s}}\right).$$
We divide the computation as follows:
\begin{align*}I_h=&\frac{1}{h}\left(\int_{\rr^n\setminus C_R^{\nu_t}(x_t)}\frac{
\tilde{\chi}_{E_t}(y)}{|x_{t}-y|^{n+s}}
- \int_{\R^n\setminus C_R^{\nu_{t-h}}(x_{t-h})}\frac{
\tilde{\chi}_{E_{t}}(y)}{|x_{t-h}-y|^{n+s}}\right)\hbox{  and }\\ 
II_h=&\frac{1}{h}\left(\int_{\rr^n\setminus C_R^{\nu_{t-h}}(x_{t-h})}\frac{
\tilde{\chi}_{E_t}(y)}{|x_{t-h}-y|^{n+s}}
- \int_{\R^n\setminus C_R^{\nu_{t-h}}(x_{t-h})}\frac{
\tilde{\chi}_{E_{t-h}}(y)}{|x_{t-h}-y|^{n+s}}\right).\end{align*}

 For the first integral  we consider
  a function $\phi^\epsilon\in C^{\infty}_0$ that
approximates $\tilde \chi_{E_t}$. Then,
\begin{equation}\label{KJ:lk:0}
I_h=\lim_{\epsilon\to0}
I_h^\epsilon,\end{equation}
with
\begin{equation}\label{KJ:lk}\begin{split}
I_h^\epsilon:=&\frac{1}{h}\left( \int_{\R^n\setminus C_R^{\nu_t}(x_t)}\frac{
\phi^\epsilon(y)}{|x_{t}-y|^{n+s}}
- \int_{\R^n\setminus C_R^{\nu_{t-h}}(x_{t-h})}\frac{
\phi^\epsilon(y)}{|x_{t-h}-y|^{n+s}}\right)\\
=&\frac{1}{h} \int_{\R^n\setminus C_R^{\nu_t}(0)}\frac{
\phi^\epsilon(y+x_t)-\phi^\epsilon(\calR_{t, t-h}y+x_{t-h})}{|y|^{n+s}}dy\\
=&
\int_{\R^n\setminus C_R^{\nu_t}(0)} \left[ \int_0^1 \frac{ \nabla\phi^\epsilon
\big(y_{h,l}\big)\cdot \delta_h
}{|y|^{n+s}}\,d\ell
\right]\,dy,
\end{split}\end{equation}
where
\begin{equation}\label{LKJH:A90}\begin{split}&
\delta_h:=\frac{ x_t-x_{t-h}+\partial_{\ell} \calR_{t, t-(1-\ell)h} y}{h},
  \\ &\calR_{t,\tau} \hbox{ is given by \eqref{rotation}} \\
&\qquad{\mbox{ and }} y_{h,l}=\calR_{t,t-(1-\ell)h} y+x_{t-h}+\ell(x_t-x_{t-h}).\end{split}\end{equation}
{F}rom Proposition \ref{derivative of rotation} we have $\partial_{\ell} \calR_{t,t-(1-\ell)h}y=h\left.\left[ (y\cdot v_i(t)) \partial_\tau v_i(\tau)+( y\cdot \nu_t) \partial_\tau \nu_\tau\right]\right|_{\tau=t-(1-\ell)h}$

Moreover, if we denote by $\calR^{-T}_{t-(1-\ell)h,t}$ the inverse of the transpose of $\calR_{t,t-(1-\ell)h}$ we have
\begin{eqnarray*} &&\textrm{div}_y \left( \frac{
\phi^\epsilon
( y_{h,l}) \calR^{-T}_{t-(1-\ell)h,t}\delta_h}{|y|^{n+s}}\right)\\
&=&
\frac{
\calR_{t,t-(1-\ell)h}\nabla\phi^\epsilon
\big(y_{h,l}\big)\cdot R^{-T}_{t-(1-\ell)h,t}\delta_h}{|y|^{n+s}}
+ \phi^\epsilon
\big( y_{h,l}\big)\,
\textrm{div}_y
\frac{\delta_h}{|y|^{n+s}}
\\
&=&
\frac{
\nabla\phi^\epsilon
\big(y_{h,l}\big)\cdot\delta_h}{|y|^{n+s}}
+ \phi^\epsilon
\big( y_{h,l}\big)\,
\textrm{div}_y\left(
\frac{\delta_h}{|y|^{n+s}}\right)
,\end{eqnarray*}
and so the divergence theorem gives that
\begin{eqnarray*}
&& 
\int_{\partial  C_R^{\nu_t}(0)}
\frac{\phi^\epsilon
( y_{h,l}) \calR^{-T}_{t-(1-\ell)h,t}\delta_h}{|y|^{n+s}}\cdot\nu_{ C_R^{\nu_t}(0)}
\,d{\mathcal{H}}^{n-1}(y)\\
&=&
\int_{\R^n\setminus  C_R^{\nu_t}(0)}
\frac{
\nabla\phi^\epsilon
\big(y_{h,l}\big)\cdot\delta_h}{|y|^{n+s}}\, dy
+\int_{\R^n\setminus  C_R^{\nu_t}(0)} \phi^\epsilon
\big( y_{h,l}\big)\,
\textrm{div}_y\left(
\frac{\delta_h}{|y|^{n+s}}\right)\,dy.
\end{eqnarray*}
We insert this information into~\eqref{KJ:lk}
and we obtain that
\begin{eqnarray*}
I_h^\epsilon &=&-
 \int_0^1
\left[
\int_{\R^n\setminus  C_R^{\nu_t}(0)} \phi^\epsilon
\big( y_{h,l}\big)\,
\textrm{div}_y\left(
\frac{\delta_h}{|y|^{n+s}}\right)\,dy.
\right]\,d\ell\\
\\&&
+\int_0^1\left[
\int_{\partial  C_R^{\nu_t}(0)}
\frac{\phi^\epsilon
( y_{h,l}) \calR^{-T}_{t-(1-\ell)h,t}\delta_h}{|y|^{n+s}}\cdot\nu_{ C_R^{\nu_t}(0)}
\,d{\mathcal{H}}^{n-1}(y))\right]\,d\ell.
\end{eqnarray*}
Thus, by~\eqref{KJ:lk:0},
\begin{equation}\label{1313}\begin{split}
I_h\;&= -\int_0^1
\left[
\int_{\R^n\setminus  C_R^{\nu_t}(0)} \chi_{E_t}
\big( y_{h,l}\big)\,
\textrm{div}_y
\left(\frac{\delta_h}{|y|^{n+s}}\right)\,dy.
\right]\,d\ell\\
&
+\int_0^1\left[
\int_{\partial  C_R^{\nu_t}(0)}
\frac{\chi_{E_t}
( y_{h,l}) \calR^{-T}_{t-(1-\ell)h,t}\delta_h}{|y|^{n+s}}\cdot\nu_{ C_R^{\nu_t}(0)}
\,d{\mathcal{H}}^{n-1}(y))\right]\,d\ell,
\end{split}\end{equation}
where  $\nu_{ C_R^{\nu_t}(0)}$ is the unit normal to the cylinder at $y$.
Now we observe that
$$
\chi_{E_t}\big(\calR_{t, t-(1-l)h} y+x_{t-h}+\ell(x_t-x_{t-h})\big)
-
\chi_{E_t}\big( y+x_{t-h}\big)=
\chi_{E_t}\big( y+x_{t-h}+O(h)\big)
-
\chi_{E_t}\big( y+x_{t-h}\big),$$
so this function is supported in a neighborhood of
size~$O(h)$ of a smooth surface. This fact, \eqref{1313}
and the integrability
of the kernel~$|y|^{-n-s}$ at infinity give that
\begin{eqnarray*}
I_h&=&
-\int_0^1
\left[
\int_{\R^n\setminus  C_R^{\nu_t}(0)} \chi_{E_t}
\big( y+x_{t-h}\big)\,
\textrm{div}_y
\left(\frac{\delta_h}{|y|^{n+s}}\right)\,dy.
\right]\,d\ell\\
\\&&
+\int_0^1\left[
\int_{\partial  C_R^{\nu_t}(0)}
\frac{\chi_{E_t}
( y+x_{t-h}) \calR^{-T}_{t-(1-\ell)h,t}\delta_h}{|y|^{n+s}}\cdot\nu_{ C_R^{\nu_t}(0)}
\,d{\mathcal{H}}^{n-1}(y))\right]\,d\ell+o(1),
\end{eqnarray*}
as~$h\to0$. 
Recalling~\eqref{LKJH:A90} and Proposition \ref{derivative of rotation}, we have  for $\tau=t-(1-l)h$ that

\begin{align*}\textrm{div}_y
\left(\frac{\delta_h}{|y|^{n+s}}\right)=&\frac{v_i(t)\cdot\partial_\tau v_i(\tau)+\nu_t
\cdot \partial_\tau \nu_\tau}{|y|^{n+s+2}} -(n+s) \frac{ y\cdot(x_t-x_{t-h}+\partial_{\ell} \calR_{t, t-(1-\ell)h} y)}{|y|^{n+s+2}h}\\
\to&
 -(n+s) \frac{ y\cdot \left(\partial_t x_t+( y\cdot v_i(t)) \, \partial_t v_i(t)+ (y\cdot \nu_t)  \partial_t \nu_t
\right)}{|y|^{n+s+2}}\hbox{ as $h\to 0$,  }\\
\hbox{ and }\calR^{-T}_{t-(1-\ell)h,t}\ \delta_h=&\calR^{-T}_{t-(1-\ell)h,t}\frac{ x_t-x_{t-h}+\partial_{\ell} \calR_{t, t-(1-\ell)h} y}{h} \\
\to& \partial_t x_t+ (y\cdot v_i(t))\,  \partial_t v_i(t)+( y\cdot \nu_t ) \,  \partial_t \nu_t
 \hbox{ as }h\to 0
.\end{align*}
Additionally,  from \eqref{evolution of v} and \eqref{evolutionnormal} we have that
\begin{align*}y\cdot[ (y\cdot v_i(t))\,  \partial_t v_i(t)+  (y,\cdot \nu_t ) \partial_t \nu_t
]=&
-( y\cdot v_i) ( \nabla^\Gamma H_s\cdot  v_i (t)) \, (y\cdot \nu_t)   +  (y\cdot  \nu_t)  ( y\cdot   \nabla^\Gamma H_s) \\ =&- (\nabla^\Gamma H_s \cdot  y^T) \, (  y\cdot  \nu_t) +( y\cdot \nu) ( y\cdot   \nabla^\Gamma H_s) \\=&0.\end{align*}
Hence,
\begin{align}
\label{LKG67lk}
\lim_{h\to0} 
I_h=(n+s) & \int_{\R^n\setminus  C_R^{\nu_t}(0)}
\tilde \chi_{E_t}\big( y+x_{t}\big)\frac{\,y
\cdot \partial_t x}{|y|^{n+s+2}}\,dy  \\ \notag
&+\int_{\partial  C_R^{\nu_t}(0)}
\tilde \chi_{E_t}\big( y+x_{t}\big)\frac{\,\left( \partial_t x_t+(y\cdot v_i(t))\,  \partial_t v_i(t)+  (y,\cdot \nu_t ) \partial_t \nu_t
\right)\cdot\nu_{
 C_R^{\nu_t}(0)}}{|y|^{n+s}}
\,d{\mathcal{H}}^{n-1}(y).
\end{align}

Now we
notice that
$$-(n+s) \frac{ y\cdot \partial_t x }{|y|^{n+s+2}}=
\textrm{div}_y
\left(\frac{ \partial_t x}{|y|^{n+s}}\right).$$
Then   using the divergence theorem we have
\begin{align*} (n+s)\int_{\R^n\setminus  C_R^{\nu_t}(0)}&
\tilde \chi_{E_t}\big( y+x_{t}\big)\frac{\,y
\cdot \partial_t x}{|y|^{n+s+2}}\,dy=\\
&\int_{E_t\setminus  C_R^{\nu_t}(x_t)}\textrm{div}_y
\left(\frac{ \partial_t x}{|y-x_t|^{n+s}}\right)dy -
\int_{E_t^c\setminus  C_R^{\nu_t}(x_t)}\textrm{div}_y
\left(\frac{ \partial_t x}{|y-x_t|^{n+s}}\right)dy\\
=&2\int_{\partial E_t\setminus  C_R^{\nu_t}(x_t)} \frac{ \nu_{\partial E_t}(y)\cdot\partial_t x}{|y-x_t|^{n+s}}\,d{\mathcal{H}}^{n-1}(y)-\int_{\partial  C_R^{\nu_t}(x_t)}  \chi_{E_t}\big( y\big)\frac{ \nu_{\partial C_R^{\nu_t}}(y)\cdot\partial_t x}{|y-x_t|^{n+s}}\,d{\mathcal{H}}^{n-1}(y),\end{align*}
where $\nu_{\partial E_t}(y)$ denotes the unit normal to $\partial E_t$ at $y$.

Plugging this into~\eqref{LKG67lk}  we obtain
\begin{align}\label{JK:AK78KK}
\lim_{h\to0}
I_h=& 
2\int_{\partial E_t\setminus  C_R^{\nu_t}(x_t)} \frac{ \nu_{\partial E_t}(y)\cdot\partial_t x}{|y-x_t|^{n+s}}\,d{\mathcal{H}}^{n-1}(y)\\ \notag & -\int_{\partial  C_R^{\nu_t}(x_t)}  \chi_{E_t}\big( y\big)\frac{ \nu_{\partial C_R^{\nu_t}}(y)\cdot\left((y-x_t)^i \partial_t v_i(t)+(y-x)^n\partial_t \nu(x)
\right)}{|y-x_t|^{n+s}}\,d{\mathcal{H}}^{n-1}(y).\end{align}

Now we
notice that,
from the definition of $C_R(0)$, the normal~$\nu_{C_R(0)}$ is either
on the tangent plane at $x_t$ (for the sides of the cylinder) or
it is parallel to the normal at $x_t$ (at the top and
the bottom of the cylinder).
Hence, at the top and bottom of the cylinder we have~
$\pm \nu_{\partial C_R^{\nu_t}}(y)\cdot \partial_t v_i(t)=- \nabla^{\Gamma} H_s\cdot  v_i(t) $  and
$\nu_{\partial C_R^{\nu_t}}(y)\cdot \partial_t \nu_t=0$, while
along the sides of the cylinder $\nu_{\partial C_R^{\nu_t}}(y)\cdot \partial_t v_i(t)=0$  and
$\nu_{\partial C_R^{\nu_t}}(y)\cdot \partial_t \nu_t=\frac{(y-x_t)^T}{|(y-x_t)^T|}\cdot \nabla^{\Gamma} H_s$.
In addition, $\tilde \chi_{E_t}=-1$ on the bottom
of the cylinder and~$\tilde \chi_{E_t}=1$ on the top.
As a consequence,
\begin{align*}
\int_{\partial  C_R^{\nu_t}(x_t)}  & \chi_{E_t}\big( y\big)\frac{ \nu_{\partial C_R^{\nu_t}}(y)\cdot\left((y-x_t)^i \partial_t v_i(t)+(y-x)^n\partial_t \nu_t\right)}{|y-x_t|^{n+s}}\,d{\mathcal{H}}^{n-1}(y)
=\\ &-2 \int_{S^{n-2}}\int_0^1R^{n-n-s}\frac{\rho^{n-2}   \nabla^{\Gamma} H_s(t)\cdot  \omega }{(\rho^2 + 1)^{\frac{n+s}2} }\,d\rho\, d\omega\\
&+\int_{S^{n-2}}\int_{-1}^1 \chi_{E_t}\big( x_t+RM_{x_t}\omega +Rz\nu_t\big)
R^{n-n-s}\frac{z M_{x_t}\omega\cdot\nabla^\Gamma H_s(x_t)}{(1+z^2)^{
\frac{n+s}{2}}} dz d\omega .\end{align*}
By symmetry the first term is 0 and 
$$\int_{S^{n-2}}\int_{-1}^1 \chi_{\Pi_t}\big( x_t+RM_{x_t}\omega +Rz\nu_t(x_t)\big)
R^{n-n-s}\frac{z M_{x_t}\omega\cdot\nabla^\Gamma H_s(x_t)}{(1+z^2)^{\frac{n+s}{2}}} dz d\omega=0.$$
we obtain the first equality of  \eqref{F3}.

The second equality  may be obtained observing that 
$$-(n+s) \frac{ y\cdot\left( \partial_t x + y\cdot v_i(t)  \partial_t v_i(t)+ y\cdot \nu_t  \partial_t \nu_t\right)}{|y|^{n+s+2}}=
\textrm{div}_y
\left(\frac{ \partial_t x+(y-x_t)^i \partial_t v_i(t)+(y-x)^n\partial_t \nu_t}{|y|^{n+s}}\right).$$
For the integral defining~$II_h$, we have
$$II_h=\frac{1}{h}\int_{\rr^n\setminus C_R(0)}\frac{\tilde{\chi}_{E_t}(y+x_{t-h})-\tilde{\chi}_{E_{t-h}}(y+x_{t-h})}{|y|^{n+s}}dy.$$

Notice that the integrand is not 0 for $y+x_{t+h}\in E_t\Delta E_{t-h}$. Since we assume that  $\partial E_t$ is smooth, we may
 parameterize this neighborhood as $y=y_t+ z \nu_{\partial E_t} (y_t)$ where $y_t\in \partial E_t$. Since we assume that the sets $E_t$ are continuous in $t$, for $h$ small enough,   $E_t\Delta E_{t-h}$ is contained in this tubular neighborhood. Moreover, a Taylor expansion in $t$ yields that $$y_{t-h}=y_t-h\partial_t y_{t} +O(h^2) \hbox{ and } (y_{t-h}-y_t)\cdot \nu_{\partial E_t}(y)=-h\partial_t y_{t} \cdot \nu_{\partial E_t}(y)+O(h^2).$$
Then we have
\begin{align*}
II_h=&\frac{1}{h}\int_{\partial E_t\setminus C_R(0)}\int_{0}^{-h\partial_t y_{t} \cdot \nu_{\partial E_t}(y)+O(h^2)}\frac{2}{|y-x_{t-h}|^{n+s}} \,dz 
\,d{\mathcal{H}}^{n-1}(y) \\
&\to -2 \int_{\partial E_t\setminus C_R(0)}\frac{\partial_t y_{t} \cdot \nu_{\partial E_t}(y)}{|y-x_{t-h}|^{n+s}}  \,d{\mathcal{H}}^{n-1}(y)
\qquad\hbox{ as }h\to 0
\end{align*}
This, together with~\eqref{JK:AK78KK}, proves~\eqref{F3}.

{F}rom Proposition \ref{derivatives of h}, we have that
$$
\frac{
\rho^{-1-s}\,\partial_t h(\rho,\varphi)}{
(1+h^2(\rho,\varphi))^{\frac{n+s}{2}} } = O(\rho^{-s}),$$
which
is integrable, thus~\eqref{F6} follows directly from \eqref{eqhsing}.
Similarly, we observe that 
\begin{eqnarray*}
&&\int_{S^{n-1}}\int_{-1}^1( \chi_{E_t}+\chi_{\Pi_t})\big( x_t+RM_{x_t}\omega +Rz\nu_t(x_t)\big)
\frac{z M_{x_t}\omega\cdot\nabla^\Gamma H_s(x_t)}{(1+z^2)^{\frac{n+s}{2}}} dz d\omega\\ &&\qquad=
\int_{S^{n-1}}\int_{\min(h(R\omega),1)}^0
\frac{z M_{x_t}\omega\cdot\nabla^\Gamma H_s(x_t)}{(1+z^2)^{\frac{n+s}{2}}} dz d\omega
\end{eqnarray*}
and equation  \eqref{limit on the cylinder} follows from the fact $h(0)=0$.

Finally, equation \eqref{F4}
follows from \cite{CFSW} and the fact
that $\partial_{\omega_i} x$ is tangential.

Equation \eqref{F5} follows now
by combining \eqref{F2} and \eqref{F3} and taking $R\to 0$
(another proof of~\eqref{F5} can be obtained using
formula~(B.2) of~\cite{Dav}; using Lemmata~A.2 and~A.4 there,
one also obtains an expansion of the quantity in~\eqref{F5}
as $s$ approaches~$1$).
\end{proof}

\begin{remark}
An equation analogous to \eqref{F5} was obtained in \cite{Dav} in a different context. Their results imply that
\begin{align*} s(1-s)\textrm{P.V.} \int_{\partial E_t}\frac{H_s(y)-H_s(x)}{|x-y|^{n+s}} dy&\to \Delta_{ \partial E_t} H \hbox{ as } s\to 1\\
s(1-s)\textrm{P.V.} \int_{\partial E_t}\frac{1-\nu(y)\cdot\nu(x)}{|x-y|^{n+s}} dy&\to |A|^2 \hbox{ as } s\to 1,\end{align*}
which recovers the classical evolution for the mean curvature $H$ under evolution by mean curvature flow.
\end{remark}




\section{Preservation of the fractional mean curvature}\label{for}

In this section we 
 show that the geometric flow preserves the positivity
of the fractional mean curvature. We need the following lemma that 
 excludes the possibility of compact hypersurfaces
with fractional
mean curvature equal to zero (we state the result
for smooth sets for the sake of simplicity):

\begin{lemma}\label{S.L}
There exists no compact hypersurface with 
with~$C^2$-boundary and vanishing fractional
mean curvature.
\end{lemma}

\begin{proof} The proof is based on a sliding method.
Roughly speaking, we take 
let~$E\subset\R^n$ be a bounded
set with~$C^2$-boundary and such that~$H_s(x,E)=0$ for any~$x\in
\partial E$, we consider
a plane of given
normal direction~$\omega$, we
slide it from infinity till it touches~$E$, and then we compare
the fractional mean curvatures at a touching point to obtain the desired
result. The details of the proof go as follows. We suppose that
\begin{equation}\label{E-02}E\ne\varnothing.\end{equation}
Let
\begin{eqnarray*}
&&\Pi_M:=\{ x\in\R^n,\ x\cdot \omega < M \}\\
{\mbox{and }}&& M_*:=\inf\{ M, \ E \subset \Pi_M\}.\end{eqnarray*}
Notice that~$M_*\in\R$, thanks to~\eqref{E-02}
and the boundedness of $E$.
In addition, $E$ is a subset of~$\Pi_{M_*}$ and
there exists~$x_t
\in (\partial E)\cap (\partial\Pi_{M_*})$.
We claim that~$E=\Pi_{M_*}$ (up to negligible subsets lying on the boundary,
and this will end the proof of Lemma~\ref{S.L}).
Indeed, if not, the positivity set of the function
$$ \tilde\chi_E -\tilde\chi_{\Pi_{M_*}}=2\chi_{\Pi_{M_*}\setminus E}$$
would have positive measure and therefore
\begin{eqnarray*}
0<\int_{\R^n} \frac{\tilde\chi_E (y)-\tilde\chi_{\Pi_{M_*}}(y)}{|x_t
-y|^{n+s}}\,dy
=H_s(x_t
,E)-H_s(x_t
,\Pi_{M_*})=0-0,
\end{eqnarray*}
and this is a contradiction.
\end{proof}

\begin{theorem}\label{Pos}
Let~$E_t$
be a compact solution of~\eqref{fmcfgeneral}. Assume that $H_s$ is differentiable and that $E_0$ has strictly positive fractional
mean curvature. Then, $E_t$  has strictly positive fractional
mean curvature for every~$t\in(0,T)$.
\end{theorem}

\begin{proof} Suppose the contrary. Then, if~$E=E_t$ is the evolving
surface, we have that~$H_s(x,E_t)>0$ for any~$x\in \partial E_t$
and any~$t\in (0,\bar t)$, but
\begin{equation}\label{X00}
H_s(\bar x,E_{\bar t})=0,\end{equation}
for some~$\bar x\in\partial E_{\bar t}$, with~$\bar t\in(0,T)$.

Notice that~$x_t\in\partial E_t$ and the function
$$ t\mapsto H_s(x_t,E_t)$$
attains its minimum in the interval~$[0,\bar t]$
and the endpoint~$\bar t$ and therefore~$
\partial_t H_s(x_t,E_t)\big|_{t=\bar t}\le 0$. Since it is also a
spatial critical point for $H_s$, we have that $\nabla H_s(\bar{x},E_t)\big|_{t=\bar t}=0$. From \eqref{F5}  in 
Theorem~\ref{derivatives non local quantities} and~\eqref{fmcfgeneral}
we obtain that
$$\partial_t H_s(\bar{x}, \bar{t})=s(1-s) \int_{\partial E_t}\frac{H_s(y)}{|y-x_{t-h}|^{n+s}}\,
d{\mathcal{H}}^{n-1}(y)  \geq 0.$$
However, since $
\partial_t H_s(x_t,E_t)\big|_{t=\bar t}\le 0$ we have that   $H_s(y)\equiv 0$, which due  to   Lemma \ref{S.L} contradicts the compactness of  $E_t$.
\end{proof}

Following the same proof as above,
we can also take into account
the case of a non-compact solution, as described in the following result.

\begin{theorem}\label{Pos2}
Let~$T>0$ and~$E_t$
be a solution of~\eqref{fmcfgeneral} in~$[0,T]$. Assume that $H_s$ is differentiable, that $E_0$ has strictly positive fractional
mean curvature and that $\partial E_t$ is uniformly spatially $C^2$ in $[0,T]$. Then, $E_t$  has strictly positive fractional
mean curvature for every~$t\in[0,T]$.
\end{theorem}

\begin{proof}

Proceeding as in the proof of the previous result,
we can also show that an alternative holds true,
namely either $E_t$  has strictly positive fractional
mean curvature for every~$t\in[0,T]$ or there is a $t_0$ such that $E_t$ 
has vanishing fractional mean curvature
for every $t\geq t_0$.

Now we show that $H_s$ cannot become identically 0. 
For this, up to a dilation,
we take a scale for which the evolving surface is locally
a smooth graph in balls of radius~$2$ centered at the surface.
Let $\phi$ be a nonnegative function supported  in the unit ball $B_1$ and $\phi\equiv 1$ in $B_{\frac{1}{2}}$. Fix $x_t=x(0,t)\in \partial E_t$ and $\epsilon>0$.
Consider the function $v:\rr^n\times [0,T)$ defined $$v(y,t)=e^{C_1 t}
\left( \frac{H_s(y)}{s\,(1-s)}+\epsilon\right)- \delta e^{-C_2 t}\phi(y-x_t), $$ where $\delta$ is chosen
such that $v(y,0)>0$ and $C_1,C_2$ are real constant to be determined. Notice that  $\delta$ can be chosen independently of $\epsilon>0$

Using equations \eqref{F5} and~\eqref{simplified eq},
and denoting by $\nu_t$ the normal at $x_t$,
we have, for $y\in \partial E_t$,
\begin{align*}& \partial_t v(y,t)\\ =\;& C_1e^{C_1 t}\left(
\frac{H_s(y)}{s\,(1-s)}+\epsilon\right)+e^{C_1 t}\left(2\textrm{P.V.}  \int_{\partial E_t} \frac{H_s(z)-H_s(y)
}{|z
-y|^{n+s}}dz+ 2H_s(y)\textrm{P.V.}  \int_{\partial E_t} \frac{1-\nu(z)\cdot\nu(y)
}{|z
-y|^{n+s}}dz \right)\\ &\qquad+C_2\delta e^{-C_2 t}\phi(y-x_t)+
H_s(x_t) \delta e^{-C_2 t}\nu_t\cdot\nabla \phi(y-x_t)\\
=\;& C_1e^{C_1 t}\left(\frac{H_s(y)}{s\,(1-s)}
+\epsilon\right)+2s(1-s)\left(\textrm{P.V.}  \int_{\partial E_t} \frac{v(z,t)-v(y,t)
}{|z
-y|^{n+s}}dz+ e^{C_1t}\,H_s(y)\textrm{P.V.}  \int_{\partial E_t} \frac{1-\nu(z)\cdot\nu(y)
}{|z
-y|^{n+s}}dz \right)\\ &+2s\,(1-s)\delta e^{-C_2 t}\textrm{P.V.}  \int_{\partial E_t} \frac{\phi(z-x_t)-\phi(y-x_t)
}{|z
-y|^{n+s}}dz+C_2\delta e^{-C_2 t}\phi(y-x_t)\\&+
H_s(x_t) \delta e^{-C_2 t}\nu_t\cdot\nabla \phi(y-x_t)
.\end{align*}

Now we claim that
\begin{equation}\label{UGU}
v(y,t)\ge 0.
\end{equation}
Since this holds for~$t=0$ (as long as~$\delta$ is sufficiently small),
to prove~\eqref{UGU} we can argue by contradiction
and
assume that there is a first time $\bar{t}$ and a point $\bar{y}$ such that $v(\bar{y},\bar{t})=0$. Such a point  is a local minimum and it holds that 
\begin{eqnarray*}
&&\partial_t v(\bar{y},\bar{t})\leq 0,  \\
&&\textrm{P.V.}  \int_{\partial E_{\bar{t}}} \frac{v(z,\bar{t})-v(\bar{y},\bar{t})
}{|z
-\bar{y}|^{n+s}}dz\geq 0\\
{\mbox{and }}&&
e^{C_1 \bar{t}}\left(\frac{H_s(\bar{y})}{s\,(1-s)}
+\epsilon\right)= \delta e^{-C_2 \bar{t}}\phi(\bar{y}-x_{\bar{t}}).\end{eqnarray*}
Hence, we have
\begin{equation}\label{MS:S:A}
\begin{split}
0\geq\partial_t v(\bar{y},\bar{t})\ge 
& \;C_1 \delta e^{-C_2 \bar{t}}\phi(\bar{y}-x_{\bar{t}})+
2s\,(1-s)\delta e^{-C_2 \bar{t}}\textrm{P.V.}  
\int_{\partial E_{\bar{t}}} \frac{\phi(z-x_{\bar{t}})-
\phi(\bar{y}-x_{\bar{t}})
}{|z
-\bar{y}|^{n+s}}dz\\ & +C_2\delta e^{-C_2 \bar{t}}\phi(\bar y-x_{\bar{t}})+
H_s(x_{\bar{t}}) \delta 
e^{-C_2\bar{t}}\nu_{\bar{t}}\cdot\nabla \phi(\bar y-x_{\bar{t}}).\end{split}\end{equation}
Now we claim that 
\begin{equation}\label{bar - y}
|\bar y-x_{\bar t}|<1.
\end{equation}
To this end, we argue by contradiction and suppose that~$|\bar y-x_t|\ge1$.
Then, using~\eqref{MS:S:A} and the assumption on the support of~$\phi$,
we find that
$$0\geq
2s\,(1-s)\delta e^{-C_2 \bar{t}}\textrm{P.V.}
\int_{\partial E_{\bar{t}}} \frac{\phi(z-x_{\bar{t}})}{|z
-\bar{y}|^{n+s}}dz >0.$$
This is a contradiction and so~\eqref{bar - y}
is proved.

Now we improve~\eqref{bar - y}, by showing that there exists~$\epsilon_0\in(0,1)$
such that
\begin{equation}\label{bar - y -eps0}
|\bar y-x_{\bar t}|<1-\epsilon_0.
\end{equation}
Again, we argue by contradiction and suppose that~$|\bar y-x_{\bar t}|\in [1-\epsilon_0,1)$.
Since~$\phi$ is smooth and vanishes along~$\partial B_1$,
we have that~$\phi(\bar y-x_{\bar t}) + |\nabla\phi(\bar y-x_{\bar t})|\le C\epsilon_0$,
for some~$C>0$. Hence, using~\eqref{MS:S:A},
and taking~$K>0$ such that
\begin{equation}\label{KAKA}
{\mbox{$H_s(x)\leq K$ for every  $x\in\partial E_t$,}}\end{equation} we see that
\begin{eqnarray*}
0 &\geq& 
2s\,(1-s) \delta e^{-C_2 \bar{t}}\textrm{P.V.}
\int_{\partial E_{\bar{t}}} \frac{\phi(z-x_{\bar{t}})-
C\epsilon_0}{|z-\bar{y}|^{n+s}}dz -CK\delta
e^{-C_2\bar{t}}\epsilon_0
\\
&\geq& \delta e^{-C_2 \bar{t}}\left[
2s\,(1-s) \textrm{P.V.}
\int_{(\partial E_{\bar{t}})\cap B_{1/2}(x_{\bar t})} 
\frac{1-C\epsilon_0}{|z-\bar{y}|^{n+s}}dz
-CK\epsilon_0
\right].\end{eqnarray*}
So we multiply by~$\delta^{-1} e^{C_2 \bar{t}}$ and, if~$\epsilon_0$
is small enough, we find that
\begin{eqnarray*} 0&\ge& s\,(1-s) \textrm{P.V.}
\int_{(\partial E_{\bar{t}})\cap B_{1/2}(x_{\bar t})} 
\frac{dz}{|z-\bar{y}|^{n+s}} -CK\epsilon_0
\\ &\ge& 2^{-n-s}s\,(1-s) 
{\mathcal{H}}^{n-1} \big( (\partial E_{\bar{t}})\cap B_{1/2}(x_{\bar t})\big)
-K\epsilon_0.\end{eqnarray*}
The smoothness of the surface gives that
$$ {\mathcal{H}}^{n-1} \big( (\partial E_{\bar{t}})\cap B_{1/2}(x_{\bar t})\big)\ge c_0$$
for some~$c_0>0$. The last two inequalities easily give a contradiction
if~$\epsilon_0$ is small enough, and so we have established~\eqref{bar - y -eps0}.

Now we set~$r_0:=1-\epsilon_0$ and we choose~$C_1$ large enough, such that
\begin{equation}\label{KAA2}
C_1\ge \frac{K\sup_{B_1}|\nabla\phi|}{\inf_{B_{r_0}}\phi},\end{equation}
where~$K$ is as in~\eqref{KAKA}.

Notice that, by~\eqref{bar - y -eps0} and~\eqref{KAA2},
\begin{equation}\label{8uj66GA}
C_1 \delta e^{-C_2 \bar{t}}\phi(\bar{y}-x_{\bar{t}})+
H_s(x_{\bar{t}}) \delta
e^{-C_2\bar{t}}\nu_{\bar{t}}\cdot\nabla \phi(\bar y-x_{\bar{t}})\ge0.\end{equation}
Let also~$C_2$ so large that
$$ C_2 \ge
\sup_{y\in B_{r_0}(x_{\bar t})}\frac{
2s\,(1-s)\delta e^{-C_2 \bar{t}} }{\phi(\bar{y}-x_{\bar{t}})}\textrm{P.V.}
\int_{\partial E_{\bar{t}}} \frac{
\phi(\bar{y}-x_{\bar{t}})-
\phi(z-x_{\bar{t}})
}{|z
-\bar{y}|^{n+s}}dz.$$
In this way, and using again~\eqref{bar - y -eps0},
$$ 2s\,(1-s)\delta e^{-C_2 \bar{t}}\textrm{P.V.}
\int_{\partial E_{\bar{t}}} \frac{\phi(z-x_{\bar{t}})-
\phi(\bar{y}-x_{\bar{t}})
}{|z
-\bar{y}|^{n+s}}dz +C_2\delta e^{-C_2 \bar{t}}\phi(\bar y-x_{\bar{t}})\ge0.$$
Then, we plug this information and~\eqref{8uj66GA}
into~\eqref{MS:S:A} and we obtain a contradiction.
This proves~\eqref{UGU}.

Then we take~$y=x_t$
and send~$\epsilon\to0$ in~\eqref{UGU} and we obtain that~$H_s(x_t)$
remains positive.
\end{proof}

\section{Estimates for entire graphs}\label{EGS}

In this section we assume that the surface is an entire graph with 
linear growth at infinity. 
That is, the surface can be parameterized by $(x,u(x,t))$ and 
\begin{equation}\label{INF:R5}
\sup_{{t\in[0,T)}\atop{|x|\ge R}
}|D u(x,t)|\le C,\end{equation}
for some~$C$, $R>0$.
Moreover, $u$ satisfies 
$$\partial_tu=-\sqrt{1+|Du|^2}\; H_s (E_u).$$

\begin{theorem} \label{height function estimate}
Let $\nu$ be the normal vector of a 
graphical surface evolving by \eqref{fmcfgeneral} and $e$ any fixed vector. 
Let
$v=(e\cdot \nu)^{-1}$, then
$$v(x,t) \leq \sup\left\{\sup_y v(y,0),\, C\right\},$$
where $C$ is such that 
\begin{equation}\label{INF:R6}
\limsup_{|x|\to \infty} v(x,t) <C,\end{equation}
for all~$t\in[0,T)$.
\end{theorem}

\begin{proof}
Let us assume that the surface is parameterized according to \eqref{simplified eq} and $\nu$ satisfies \eqref{evolutionnormal}. Then
$$\partial_t v=-v^2(e\cdot\nabla^{\Gamma} H_s).$$
{F}rom Theorem~\ref{derivatives non local quantities} we have that
$$e\cdot\nabla^{\Gamma} H_s= (n+s)s(1-s)  \textrm{P.V.}\int_{\rr^n} \tilde{\chi}_{E_t}(y)\frac{(y-x)\cdot e^T}{|x-y|^{n+s+2}}dy, 
$$
where $e^T$ is the tangential component of $e$ at $x_t
$.

Noticing that $(n+s)\frac{(y-x)\cdot e^T}{|x-y|^{n+s+2}}=-\textrm{div}_y\left(\frac{e^T}{|x-y|^{n+s}}\right)$, it follows from the divergence theorem that 
$$e\cdot\nabla^{\Gamma} H_s=2 s(1-s)\int_{\partial E_t}\frac{e^T\cdot \nu(y)}{|x-y|^{n+s}}.$$
Since $e^T=e-v^{-1}(x)\nu(x)$, it holds that $e^T\cdot \nu(y)=v^{-1}(y)-v^{-1}(x)\nu(x)\cdot \nu(y)$.
Then, if $v$ attains a maximum at $x$, we have that $e^T\cdot \nu(y) \geq 0$  (and similarly  $e^T\cdot \nu(y) \leq 0$ at minima).
We may conclude from the maximum principle that $v$ does not have interior maxima (resp. minima).
\end{proof}

Noticing that  for an evolving graph it holds that
$$(e_n\cdot \nu)^{-1}=\sqrt{1+|Du|^2},$$ and thus~\eqref{INF:R5}
implies~\eqref{INF:R6}, we have
\begin{corollary}
$|Du|$ is uniformly bounded in time.
\end{corollary}

\begin{theorem}
Let $v=\sqrt{1+|Du|^2}$, then 
the quantity $ v H_s$ is uniformly bounded in terms of the initial condition.
\end{theorem}

\begin{proof}
Considering the set $\Pi$ as the epigraph of the plane $z=u(x_t
,t)+\nabla u(x_t
,t)\cdot(x-x_t
)+u(x_t
,t)$,
 we may write
\begin{align*}H_s(x_t
,E)=&  s(1-s) \int_{\rr^{n-1}}\int_{u(x,t)}^{\nabla u(x_t
,t)\cdot(x-x_t
)+u(x_t
,t)}\frac{dz}{\left((z-u(x_t
,t))^2
+|x-x_t
|^2\right)^{\frac{n+2}{2}}}dx
\\=&  s(1-s) \int_{\rr^{n-1}}\frac{1}{|x-x_t
|^{n+s-1}}\int_{\frac{u(x,t)-u(x_t
,t)}{|x-x_t
|}}^{\nabla u(x_t
,t)\cdot\left(\frac{x-x_t
}{|x-x_t
|}\right)}\frac{dz}{\left(z^2
+1\right)^{\frac{n+s}{2}}}dx
\end{align*}

Let $z_m=\frac{u(x,t)-u(x_t
,t)}{|x-x_t
|}$ and $z_M=\nabla u(x_t
,t)\cdot \frac{x-x_t
}{|x-x_t
|}.$
Then
$$\partial_t z_m=\frac{-H_s v(x,t)+ H_s v(x_t
,t)}{|x-x_t
|}$$
and 
$$\partial_t z_M=\nabla(-H_s v(x_t
,t))\cdot\frac{x-x_t
}{|x-x_t
|}.$$
As a consequence, we have
\begin{align*}\partial_tH_s(x_t
,E)=&    s(1-s) \int_{\rr^{n-1}}\frac{1}{|x-x_t
|^{n+s-1}}
\left(\frac{\partial_t z_M}{\left(z_M^2
+1\right)^{\frac{n+s}{2}}}- \frac{\partial_t z_m}{\left(z_m^2
+1\right)^{\frac{n+s}{2}}}
\right), \hbox{ and }\\
\partial_t (v H_s)=&\frac{Du\cdot D(-H_s v)}{\sqrt{1+|Du|^2}}H_s+ v \partial_tH_s(x_t
,E).
\end{align*}

Assume that a maximum (resp. minimum) point of $v H_s$ is attained at $(x_t
, t_0).$ Then  $D(-H_s v)=0$, $\partial _t z_M=0$ and $\partial_t z_m \geq 0 $ (resp. $\leq 0$) and only identically  0 if $v H_s$ is constant. Then we conclude that there are no interior maxima or minima for this quantity.
\end{proof}

\begin{remark}
The previous estimates imply that if there is decay at infinity $H_s$ remains bounded for all times.
\end{remark}

 \section{Estimates for star-shaped surfaces}\label{ESS}

We show an estimate for star-shaped surfaces that is analog to Theorem \ref{height function estimate}

\begin{theorem}
Let $v=(x\cdot \nu)^{-1}$. Then there exists $T^*>0$ such that  $v(t)\leq C$ in $[0,T^*)$, where $C$ depends on $v(0)$ and $\sup |H_s|$.
\end{theorem}

\begin{proof}

We assume like in the proof of \ref{height function estimate} that the surface is parameterized as in \eqref{simplified eq}, then
we have 
\begin{align*}\partial_t v=&-v^2(x_t\cdot \nu+x\cdot\nu_t)\\
=&v^2(H_s-x\cdot\nabla^{ \Gamma} H_s)\end{align*}

Following the computations in the proof of Theorem  \ref{height function estimate} we have that
$$x\cdot\nabla^{\Gamma} H_s=2s(1-s)\int_{\partial E_t}\frac{x^T\cdot \nu(y)}{|x-y|^{n+s}}dy.$$
since $x^T=x-x\cdot\nu(x) \nu(x)=(x-y)+(y-x\cdot\nu(x)\nu(x))$ we have 

\begin{align*}
\frac{x\cdot\nabla^{\Gamma} H_s}{s(1-s)}=&2\int_{\partial E_t}\frac{(x-y)\cdot \nu(y)}{|x-y|^{n+s}}dy+2\int_{\partial E_t}\frac{v^{-1}(y)-v^{-1}(x)\nu(x)\cdot \nu(y)}{|x-y|^{n+s}}dy\\=&s H_s+2\int_{\partial E_t}\frac{v^{-1}(y)-v^{-1}(x)\nu(x)\cdot \nu(y)}{|x-y|^{n+s}}dy.\end{align*}
Accordingly, we have
$$\partial_t v=v^2
\left((1-s^2(1-s))H_s-2s(1-s)\int_{\partial E_t}\frac{v^{-1}(y)-v^{-1}(x)\nu(x)\cdot \nu(y)}{|x-y|^{n+s}}dy\right).$$
Hence, at a spacial maximum of $v$ we have 
$$\partial_t (\max_{\mathbb S^n} v(\cdot, t))\leq (1- s^2(1-s))\max_{\mathbb S^n} v(\cdot, t)^2 H_s.$$
Then we find that
$$\max_{\mathbb S^n} v(\cdot, t)\leq \frac{\max_{\mathbb S^n} v(\cdot, 0)}{1- (1- s^2(1-s))t \max_{\mathbb S^n} v(\cdot, 0) \sup_{S^\times [0,T)} H_s}.$$
Notice that the bound can be extended as long as $H_s$ remains bounded.
\end{proof}

The previous computation yields a gradient bound and that star-shapedness is preserved:

\begin{corollary}
Assume that $f$ satisfies \eqref{s-flow-in-f}. Then, if  $H_s$ remains bounded, $|\nabla f|$ is bounded for a fixed time that depends of the initial condition and bounds of $H_s$.
\end{corollary}

\begin{proof}
Notice that $x=f\omega$ and $\nu =\frac{f\omega-\nabla f }{ \sqrt{f^2+|\nabla f|^2}}$.
Then $x\cdot \nu=\frac{f^2}{ \sqrt{f^2+|\nabla f|^2}}$.

Then $v\leq C$ is equivalent to 
$$ \sqrt{f^2+|\nabla f|^2}\leq C f^2\leq C\max f^2(\cdot,0),$$
which gives the desired result.
\end{proof}

\begin{corollary}
Assume that $E_t$ is a solution to \eqref{fmcfgeneral} and that $E_0$, then $E_t$ remains star-shaped.
\end{corollary}

\section*{Acknowledgements}
It is a pleasure to thank Eleonora Cinti, Luca Lombardini and Carlo Sinestrari
for their very interesting and useful
comments on a preliminary version of this manuscript. The first author has been partially supported by Proyecto Fondecyt Regular 1150014.
The second author has been supported by the ERC
grant 277749 ``EPSILON Elliptic Pde's and Symmetry of Interfaces and Layers for
Odd Nonlinearities'' and the
PRIN grant 201274FYK7 ``Aspetti variazionali e
perturbativi nei problemi differenziali nonlineari''.

\end{document}